\documentclass[review,hidelinks,onefignum,onetabnum,final]{siamonline220329}

\usepackage{lipsum}
\usepackage{amsfonts,bm}
\usepackage{graphicx}
\usepackage{subcaption}
\usepackage{epstopdf}
\usepackage{algorithmic}
\usepackage{tikz-cd,mathtools} 
\ifpdf
  \DeclareGraphicsExtensions{.eps,.pdf,.png,.jpg}
\else
  \DeclareGraphicsExtensions{.eps}
\fi

\newcommand{\differential}{{\rm{d}}}
\newcommand\puteqnum{
\refstepcounter{equation}\textup{(\theequation)}}

\newsiamremark{remark}{Remark}
\newsiamremark{hypothesis}{Hypothesis}
\crefname{hypothesis}{Hypothesis}{Hypotheses}
\newsiamthm{claim}{Claim}

\headers{Probabilistic Lambert Problem}{A. M. H. Teter, I. Nodozi, A. Halder}

\title{Probabilistic Lambert Problem: Connections with Optimal Mass Transport, Schr\"{o}dinger Bridge and Reaction-Diffusion PDEs\thanks{Submitted to the editors DATE.
\funding{This work was supported in part by NSF grant 2112755 and the 2023-24 Baskin School of Engineering Dissertation Year Fellowship at the University of California Santa Cruz.}}}

\author{Alexis M.H. Teter\thanks{Department of Applied Mathematics, University of California Santa Cruz, CA 
  (\email{amteter@ucsc.edu}).}
\and Iman Nodozi\thanks{Department of Electrical and Computer Engineering, University of California Santa Cruz, CA 
  (\email{inodozi@ucsc.edu}).}
\and Abhishek Halder\thanks{Department of Aerospace Engineering, Iowa State University, IA 
  (\email{ahalder@iastate.edu}).}}

\usepackage{amsopn}

\newcommand{\complexi}{{\mathrm{i}}}

\ifpdf
\hypersetup{
  pdftitle={Probabilistic Lambert Problem: Connections with Optimal Mass Transport, Schr\"{o}dinger Bridge and Reaction-Diffusion PDEs},
  pdfauthor={A. M. H. Teter, I. Nodozi, A. Halder}
}
\fi

\begin{document}
	\nolinenumbers

\maketitle

\begin{abstract}
The Lambert problem originated in orbital mechanics. It concerns with determining the initial velocity for a boundary value problem involving the dynamical constraint due to gravitational potential with additional time horizon and endpoint position constraints. Its solution has application in transferring a spacecraft from a given initial to a given terminal position within prescribed flight time via velocity control. We consider a probabilistic variant of the Lambert problem where the knowledge of the endpoint constraints in position vectors are replaced by the knowledge of their respective joint probability density functions. We show that the Lambert problem with endpoint joint probability density constraints is a generalized optimal mass transport (OMT) problem, thereby connecting this classical astrodynamics problem with a burgeoning area of research in modern stochastic control and stochastic machine learning. This newfound connection allows us to rigorously establish the existence and uniqueness of solution for the probabilistic Lambert problem. The same connection also helps to numerically solve the probabilistic Lambert problem via diffusion regularization, i.e., by leveraging further connection of the OMT with the Schrödinger bridge problem (SBP). This also shows that the probabilistic Lambert problem with additive dynamic process noise is a generalized SBP, and can be solved numerically using the so-called Schrödinger factors, as we do in this work. Our analysis leads to solving a system of reaction-diffusion PDEs where the gravitational potential appears as the reaction rate.
\end{abstract}



\section{Introduction}\label{sec:Intro}
The classical Lambert problem in orbital mechanics is a \emph{partially specified} finite dimensional two point boundary value problem over a prescribed time horizon $[t_0,t_1]$ where $t_0,t_1$ are fixed, subject to a two-body gravitational potential force field, and endpoint constraints on the relative position vector. 

Specifically, at any time $t\in[t_0,t_1]$, let $\bm{r}=(x,y,z)^{\top}\in\mathbb{R}^{3}$ be the position and $\bm{v}:=\dot{\bm{r}}\in\mathbb{R}^{3}$ be the velocity of a spacecraft with respect to the Earth centered inertial frame. Denote the Euclidean magnitude of the position vector $\bm{r}$ as $\vert \bm{r} \vert := \sqrt{x^2 + y^2 + z^2}$.

The classical Lambert problem seeks to compute a velocity field $\bm{v}=\bm{v}(\bm{r},t)$ subject to the two-body gravitational equation of motion, endpoint position and hard flight time constraints: 
\begin{align}
\ddot{\bm{r}} = -\nabla_{\bm{r}}V\left(\bm{r}\right), \quad \bm{r}(t=t_0) = \bm{r}_0\:\text{(given)}, \quad \bm{r}(t=t_1) = \bm{r}_{1}\:\text{(given)},    \label{TPBVP}    
\end{align}
where the nonlinear potential $V(\cdot)$ is bounded, and has typical form:  
\begin{align}
V(\bm{r}) := - \frac{\mu}{|\bm{r}|} - \frac{\mu J_2 R_{\textrm{Earth}}^2}{2 |\bm{r}|^3} \left( 1 - \frac{3z^2}{|\bm{r}|^2} \right), \quad \bm{r}\in [R_{\rm{Earth}}, +\infty).
\label{defPotential}    
\end{align}
The partial specification refers to that only the endpoint positions are specified in \cref{TPBVP}, i.e., the initial velocity remains to be determined.

In \cref{defPotential}, the parameter $\mu=398600.4415$ km$^{3}$/s$^{2}$ denotes the
product of the Earth's gravitational constant and mass. The parameter $J_2 = 1.08263 \times 10^{-3}$ denotes the second zonal harmonic coefficient (unitless), which is a measure of the oblateness of the Earth. The radius of Earth $R_{\textrm{Earth}} = 6378.1363$ km. Although we use the potential \cref{defPotential} in numerical simulation (\cref{sec:Numerical}) for specificity, it will be apparent that all our developments go through for more general geopotential models \cite{egm2008}, i.e., when higher order harmonics (zonal and tesseral terms) are included in $V(\bm{r})$. Specifically, all our results are valid as long as $V(\bm{r})$ is bounded and continuously differentiable, which is indeed the case for all $\bm{r}\in [R_{\rm{Earth}}, +\infty)$.

In this work, we consider a probabilistic Lambert problem that softens the endpoint constraints in \cref{TPBVP} to
\begin{align}
\bm{r}(t=t_0) \sim \rho_0\:\text{(given)}, \quad \bm{r}(t=t_1) \sim \rho_1\:\text{(given)},   
\label{EndpointPDFs}    
\end{align}
where $\sim$ is a shorthand for ``follows the statistical law''. Thus, we allow stochastic uncertainties in the endpoint relative positions, and instead of steering between two given position vectors, we now consider steering between their statistics given by the respective joint probability density functions (PDFs) $\rho_0,\rho_1$. So the probabilistic Lambert problem becomes
\begin{subequations}
\begin{align}
&\underset{\dot{\bm{r}}=\bm{v}(\bm{r},t)}{\text{find}}\quad \bm{v}\\
&\ddot{\bm{r}} = -\nabla_{\bm{r}}V\left(\bm{r}\right),\label{ProbLambertDetDyn}\\
& \bm{r}(t=t_0) \sim \rho_0\:\text{(given)}, \quad \bm{r}(t=t_1) \sim \rho_1\:\text{(given)}.\label{ProbLambertEndpoint}
\end{align}
\label{ProbLambertFeasibility}
\end{subequations}
The PDF $\rho_0$ models initial condition uncertainty (e.g., due to statistical estimation errors). The PDF $\rho_1$ on the other hand, encodes desired statistical performance specification, i.e., the allowable terminal condition uncertainty. A small terminal condition uncertainty would imply that $\rho_1$ is ``tall and skinny'' approximating Dirac delta with its Lebesgue mass supported on a small subset of $\mathbb{R}^3$. A less stringent specification of terminal uncertainty would allow $\rho_1$ to have its mass spread over a subset of $\mathbb{R}^3$ with a larger Lebesgue volume. If problem \cref{ProbLambertFeasibility} is feasible, then intuition suggests that a more (resp. less) stringent $\rho_1$ specification would result in a more (resp. less) ``expensive'' control $\bm{v}$ but it is unclear in what sense. As such, \cref{ProbLambertFeasibility} is posed as a \emph{feasibility} problem and even if it admits a unique solution $\bm{v}$ (also unclear why), it is not obvious what \emph{optimality} guarantee, if any, does that solution usher.  

Mathematically, \cref{ProbLambertFeasibility} is an \emph{infinite dimensional} two point boundary value problem since the endpoints $\rho_0,\rho_1$ are elements in the manifold of joint PDFs supported on $\mathbb{R}^{3}$. Solving \cref{ProbLambertFeasibility} amounts to computing a PDF-valued curve parameterized by $t\in[t_0,t_1]$ on this infinite dimensional manifold connecting the endpoints $\rho_0$ and $\rho_1$. Since this manifold is not a vector space, a rigorous solution of this problem requires understanding and systematically leveraging the geometry (e.g., metric structure) in this manifold. The main technical contribution of this work is to clarify how this can be done through a generalization of the theory of \emph{optimal mass transport (OMT)}. We will provide the necessary background for dynamic OMT in \cref{subsec:OMT}. Standard references on OMT are \cite{villani2003topics,villani2009optimal}.

Connecting problem \cref{ProbLambertFeasibility} with the OMT allows us to make progress on multiple fronts. \emph{First}, it allows us to rigorously establish that the solution for \cref{ProbLambertFeasibility} is indeed unique. \emph{Second}, our mathematical development for proving uniqueness reveals that the unique velocity field $\bm{v}$ is indeed optimal in certain minimum effort sense. \emph{Third}, it allows generalizing \cref{ProbLambertFeasibility} for the case when the velocity has additive process noise, i.e., when the controlled ODE \begin{align}
\dot{\bm{r}}=\bm{v}(\bm{r},t) 
\label{ODESamplePath}
\end{align}
is replaced with the It\^{o} stochastic differential equation (SDE) \begin{align}
\differential\bm{r} = \bm{v}(\bm{r},t)\:\differential t + \sqrt{2\varepsilon}\:\differential\bm{w}(t),
\label{SDESamplePath}
\end{align} 
where $\bm{w}(t)\in\mathbb{R}^3$ is the standard Wiener process, and the strength of the process noise $\varepsilon>0$ is \emph{not necessarily small}. This is of practical interest: the process noise in \cref{SDESamplePath} may result from noisy actuation in velocity command, or from stochastic disturbance in atmospheric drag, solar radiation pressure, etc. Our results show that just like problem \cref{ProbLambertFeasibility} is a generalized OMT, the corresponding problem with \cref{ODESamplePath} replaced by \cref{SDESamplePath} is a generalized \emph{Schr\"{o}dinger bridge problem (SBP)} \cite{schrodinger1931umkehrung,schrodinger1932theorie,wakolbinger1990schrodinger}. We will summarize the necessary background on SBP in \cref{subsec:SBP}.

\subsection*{History and related works}
The original formulation of the Lambert problem was deterministic. It bears the name of Johann Heinrich Lambert, who is also credited for the first proof for the irrationality of $\pi$ \cite[p. 141]{berggren1997pi} in 1761. In a letter sent to Euler in the February of the same year \cite{euler1924briefwechsel}, Lambert mentions of his result on the problem of determining the orbit of a comet relating a given flight time $t_1-t_0$ to the problem data $\bm{r}_0,\bm{r}_1$, assuming the connecting arc is a parabola. Such problems were studied earlier by Euler \cite{euler1743determinatio}. Lambert's results became well known through his book \cite{lambert1761ih} with follow-up works due to many others including Lagrange, Laplace, Gauss, and Legendre; for a detailed chronology, see \cite[Sec. 9]{albouy2019lambert}. These early investigations considered Keplerian potential, i.e., \cref{defPotential} without the $J_2$ term, in which case the feasible arcs connecting $\bm{r}_0,\bm{r}_1$ are conic sections which inspired Eulcidean geometric treatments in these studies.

The deterministic Lambert problem has a substantial modern literature in the guidance-navigation-control community -- a non-exhaustive list of references includes \cite{battin1977lambert,battin1984elegant,engels1981gravity,shen2003optimal,mcmahon2016linearized}, \cite[Ch. 13.4]{junkins2009analytical}. The sustained interest in this topic stems from its relevance in spacecraft trajectory design and rendezvous problems. A deterministic optimal control reformulation of the Lambert problem that inspires our development (see \cref{subsec:LambertOTformulation}) is due to \cite{bando2010new}.

The interest in Lambert problem with probabilistic uncertainty is relatively recent. It has been investigated before using Taylor expansion \cite{armellin2012high}, using linearization followed by second order statistics (i.e., covariance) mapping \cite{schumacher2015uncertain,zhang2018covariance}, using the pushforward mapping of probability measures  \cite{adurthi2020uncertain}, using polynomial approximation \cite{hall2018higher}, and using neural network and Gaussian process regression \cite{gueho2020comparison}. These works consider the endpoint probabilistic uncertainties as in \cref{ProbLambertEndpoint} or some approximants thereof. However, they do not explicitly account for the stochastic dynamics \cref{SDESamplePath}. In contrast, recognizing \cref{ProbLambertFeasibility} as a generalized OMT naturally leads to a further generalized SBP that precisely corresponds to the controlled stochastic dynamics \cref{SDESamplePath}. 

The reaction rate appearing in the system of coupled reaction-diffusion PDEs that we derive has the physical meaning of gravitational potential. It also has a probabilistic interpretation: the rate of killing and creation of probability mass. The recent work \cite{chen2022most} also considers an SBP with killing or creation but with \emph{unbalanced} endpoint marginals, which is not the case for us. While it is possible to convert that unbalanced problem to an equivalent balanced problem, doing so changes the mathematical structure and meaning of the stochastic optimal control problem, see \cite[eq. (4.3) and Sec. 5]{chen2022most}.

\subsection*{Contributions}
This work makes the following novel contributions.
\begin{itemize}
    \item We establish that the Lambert problem with endpoint joint PDF constraints, i.e., problem \cref{ProbLambertFeasibility}, is a generalized variant of the OMT problem. The importance of the OMT in dynamics-control problems with probabilistic uncertainties is being recognized across different research communities at a very rapid pace \cite{chen2021optimal,peyre2019computational,santambrogio2015optimal,tong2020trajectorynet,buttazzo2012optimal,ambrosio2005gradient,mei2018mean,halder2014probabilistic,lee2015performance,caluya2019gradient}. In this article, we uncover a hitherto unknown link between the OMT and the probabilistic Lambert problem.     
    \item Making a precise connection with the OMT allows us to theoretically guarantee existence-uniqueness for the solution of the probabilistic Lambert problem. We accomplish this using two mathematical ingredients: a deterministic optimal control reformulation of the Lambert problem followed by a generalization of the classical dynamic OMT. The former utilizes classical analytical mechanics while the latter relies on a relatively recent development, viz. Figalli's theory \cite{figalli2007optimal} of OMT with cost derived from an action functional.

    \item The connection with OMT also allows for generalization in the sense that the probabilistic Lambert problem \emph{with process noise} leads to a generalized SBP, which is a stochastic regularization of the OMT. We find that this generalized SBP has an additive state cost which equals to the negative of the gravitational potential. This is a scantily investigated \cite{dawson1990schrodinger,aebi1992large,7170905} variant of the generalized SBP, and this paper appears to be the first work on the same in an engineering problem. Yet, this is precisely the formulation that the probabilistic Lambert problem with process noise leads to. We note that the focus of \cite{dawson1990schrodinger,aebi1992large} were to deduce a probabilistic interpretation of such generalized SBPs in the language of the large deviation principle \cite{dembo2009large}. Our interest here is to actually solve such problems using computational algorithm, which is again a new endeavor.
    
  \item We show that the additive state cost in the generalized SBP formulation corresponding to the probabilistic Lambert problem with process noise, can be reduced to solving a system of boundary-coupled reaction-diffusion PDEs with state-dependent nonlinear reaction rates. Interestingly, the gravitational potential, with appropriate signs, play the role of these reaction rates. We propose a novel algorithm to numerically solve this system of PDEs and boundary conditions.
    
    \item We demonstrate that these newfound connections with the OMT and the SBP, allow us to numerically solve the probabilistic Lambert problem using \emph{nonparametric} computation. In particular, our approach avoids parameterizing the nonlinear dynamics (e.g., using Taylor series) or the statistics (e.g., assuming a finite number of moments such as mean and covariance, parametric families such as Gaussian mixture, exponential family). In this sense, our results are as assumption free as one might hope for.
\end{itemize}

\subsection*{Notations} We use $\vert\cdot\vert$ to denote the Euclidean magnitude and $\langle\cdot,\cdot\rangle$ to denote the Eulcidean inner product. The symbols $\nabla_{\bm{r}},\nabla_{\bm{r}}\cdot,\Delta_{\bm{r}}$ denote the Euclidean gradient, divergence and Laplacian operators with respect to the vector $\bm{r}$, respectively. We use $C_{c}^{\infty}\left(\mathcal{S}\right)$ to denote the space of infinitely differentiable functions that are supported on compact subsets of the set $\mathcal{S}$. The notation $C^{m,n}\left(\mathcal{T};\mathcal{S}\right)$ is used for the space of functions which are $m$ and $n$ times continuously differentiable on $\mathcal{T}$ and $\mathcal{S}$, respectively. Abbreviations a.s., a.e., w.r.t. respectively stand for almost surely, almost everywhere, with respect to.

\subsection*{Organization}
In \cref{sec:Prelim}, we provide a brief background on the OMT and the SBP, needed for the technical development that ensues. The subsequent sections are written in a way that a reader, if wishes so, may skip \cref{sec:Prelim} during the first pass and come back to it as needed. In \cref{sec:LambertOT}, we show that accounting for \cref{EndpointPDFs} in \cref{TPBVP} leads to a generalized version of the OMT problem, which we refer to as the \emph{Lambertian optimal mass transport (L-OMT)}. In \cref{sec:LambertSBP}, we consider a regularized version of the L-OMT, which we refer to as the \emph{Lambertian Schr\"{o}dinger bridge problem (L-SBP)}. In \cref{sec:Algorithm}, we propose an algorithm to numerically solve the L-SBP. The same algorithm can be used to solve the L-OMT in the small diffusive regularization limit. We provide illustrative numerical results in \cref{sec:Numerical}.    


\section{Background on OMT and SBP}\label{sec:Prelim}
In this Section, we summarize the classical OMT and SBP background relevant to the technical development that follows. For the Lambert problem of interest, we will need certain generalizations of the classical results summarized here. These generalizations and derivations will be done in situ.

\subsection{Optimal Mass Transport (OMT)}\label{subsec:OMT}
Let $\mathcal{P}_2\left(\mathbb{R}^{d}\right)$ denote the set of PDFs supported over $\mathbb{R}^{d}$, whose elements have finite second moments, i.e., 
$$
\mathcal{P}_2\left(\mathbb{R}^d\right):=\left\{\rho: \mathbb{R}^d \mapsto \mathbb{R}_{\geq 0} \mid \int_{\mathbb{R}^d} \rho \mathrm{d} \boldsymbol{r}=1, \int_{\mathbb{R}^d}\vert\boldsymbol{r}\vert^2 \rho \mathrm{d} \boldsymbol{r}<\infty\right\}.
$$
For $\rho_0,\rho_1\in\mathcal{P}_2\left(\mathbb{R}^{d}\right)$, let the collection of all PDF-valued trajectories $\rho(\cdot,t)$ in $\mathcal{P}_{2}\left(\mathbb{R}^{d}\right)$ that are continuous in $t\in[t_0,t_1]$ with endpoints $\rho(\cdot,t=t_0)=\rho_0, \rho(\cdot,t=t_1)=\rho_1$, i.e., 
\begin{align}
\mathcal{P}_{01}:=\{\rho(\cdot,t)\in\mathcal{P}_{2}\left(\mathbb{R}^{d}\right) \mid \rho(\cdot,t=t_0)=\rho_0,\rho(\cdot,t=t_1)=\rho_1\}.
\label{DefP01}    
\end{align}
 Let $\mathcal{V}$ be the set comprising all Markovian control policies with finite energy, i.e.,
\begin{align}
\mathcal{V}:=\bigg\{\bm{v}:\mathbb{R}^{d}\times \left[t_0,t_1\right]\mapsto\mathbb{R}^{d}\mid\int_{t_0}^{t_1}\int_{\mathbb{R}^d}\vert\bm{v}\vert^2 \rho(\bm{r},t)  \differential \bm{r}\differential t < \infty \;\text{for all}\;\rho\in\mathcal{P}_{2}\left(\mathbb{R}^{d}\right)\bigg\}.
\label{DefFerasibleControl}
\end{align}

For a given time horizon $[t_0,t_1]$, the dynamic formulation of OMT, introduced by Benamou and Brenier \cite{benamou2000computational}, is the variational problem:
\begin{subequations}
\label{BenamouBrenierOMT}  
\begin{align}
&\underset{(\rho,\bm{v})\in\mathcal{P}_{01}\times\mathcal{V}}{\arg\inf}\int_{t_{0}}^{t_{1}}\int_{\mathbb{R}^{d}}\frac{1}{2}\vert\bm{v}\vert^2\:\rho(\bm{r},t)\differential\bm{r}\differential t\label{BBobj}\\
&\qquad\dfrac{\partial\rho}{\partial t} + \nabla_{\bm{r}}\cdot\left(\rho\bm{v}\right) = 0.\label{BBdyn}
\end{align}
\end{subequations}
For details, we refer the readers to \cite[Thm. 8.1]{villani2003topics}, \cite[Ch. 8]{ambrosio2005gradient}.

The objective \cref{BBobj} seeks to minimize the average control effort $\frac{1}{2}\vert\bm{v}\vert^2$ w.r.t. the joint PDF $\rho(\bm{r},t)$. The evolution of the state PDF $\rho(\bm{r},t)$ is governed by the Liouville PDE \cref{BBdyn}, which describes the transport of mass under a feasible control policy $\bm{v}\in\mathcal{V}$. The sample path dynamics associated with the Liouville PDE \cref{BBdyn} is the controlled ODE \cref{ODESamplePath}; see e.g., \cite{halder2011dispersion}. The existence-uniqueness of solution for \cref{BenamouBrenierOMT} is guaranteed provided the endpoints $\rho_0,\rho_1\in\mathcal{P}_2\left(\mathbb{R}^{d}\right)$.


\subsection{Schr\"{o}dinger Bridge Problem (SBP)}\label{subsec:SBP}
The SBP originated from the works of Erwin Schr\"{o}dinger \cite{schrodinger1931umkehrung,schrodinger1932theorie,wakolbinger1990schrodinger}. Notably, it predates both the mathematical theory of stochastic processes and the theory of feedback control. Schr\"{o}dinger's original motivation for studying this problem was to find a probabilistic interpretation of quantum mechanics.
In recent years, SBP with nonlinear prior dynamics \cite{chen2015fast,elamvazhuthi2018optimal,haddad2020prediction,caluya2021wasserstein,caluya2021reflected,nodozi2022schr} have appeared in the literature. 

The classical formulation of SBP computes the minimum effort additive control required to steer a given PDF to another given PDF over a specified finite time horizon, subject to a controlled Brownian motion constraint. This leads to a stochastic optimal control problem:
\begin{subequations}
\begin{align}
&\underset{(\rho,\bm{v})\in\mathcal{P}_{01}\times\mathcal{V}}{\arg\inf}\int_{t_0}^{t_1}\int_{\mathbb{R}^{d}}\frac{1}{2}\vert\bm{v}\vert^2\:\rho(\bm{r},t)\:\differential\bm{r}\:\differential t \label{SBPobj}\\ &\qquad\dfrac{\partial\rho}{\partial t} + \nabla_{\bm{r}}\left(\rho\bm{v}\right) = \varepsilon\Delta_{\bm{r}}\rho,\quad\varepsilon > 0.\label{SBPdyn}
\end{align}
\label{SBP}  
\end{subequations}
As is the case for \cref{BenamouBrenierOMT}, the existence-uniqueness of the solution for \cref{SBP} is guaranteed for $\rho_0,\rho_1\in\mathcal{P}_2\left(\mathbb{R}^{d}\right)$. Moreover, the solution of \cref{SBP} enjoys time-reversibility in the sense that swapping the endpoint data $\rho_0,\rho_1$ results in a solution that is precisely the forward solution run reverse in time over the given time horizon.

When $\varepsilon\downarrow 0$, the SBP \cref{SBP} reduces to the OMT problem \cref{BenamouBrenierOMT} and its solution converges \cite{mikami2004monge,leonard2012schrodinger,leonard2014survey} to the solution of \cref{BenamouBrenierOMT}. Notice that \cref{BBdyn} is the \emph{first order} Liouville PDE while \cref{SBPdyn} is the \emph{second order} Fokker-Planck-Kolmogorov (FPK) PDE. The macroscopic dynamics \cref{SBPdyn} corresponds to the (microscopic) controlled sample path dynamics \cref{SDESamplePath}.


\section{Lambertian Optimal Mass Transport}\label{sec:LambertOT}
The starting point of our development is reformulating \eqref{TPBVP} as a standard deterministic optimal control problem.

\subsection{Formulation}\label{subsec:LambertOTformulation}
We start with a known result \cite{bando2010new,kim2020optimal} that exactly transforms problem \eqref{TPBVP} as a standard deterministic optimal control problem with state variable $\bm{r}$ and control variable $\bm{v}$, given by
\begin{subequations}
\begin{align}
&\underset{\bm{v}}{\arg\inf}\quad\displaystyle\int_{t_0}^{t_1}\left(\dfrac{1}{2}\vert \bm{v}\vert ^{2} - V(\bm{r})\right)\differential t \label{LambertOCPobj}\\  
&\dot{\bm{r}} = \bm{v}, \label{LambertOCPdynconstr}\\
& \bm{r}(t=t_0) = \bm{r}_0\:\text{(given)}, \quad \bm{r}(t=t_1) = \bm{r}_{1}\:\text{(given)}\label{LambertOCPterminalconstr}.
\end{align}
\label{LambertOCP}
\end{subequations}
Letting $(\bm{q},\dot{\bm{q}})\equiv (\bm{r},\bm{v})$, this reformulation can be interpreted as the celebrated Hamilton's principle of least action where the objective in \eqref{LambertOCPobj} is an action integral $\int_{t_0}^{t_1} L(\bm{q},\dot{\bm{q}})\:\differential t$ with Lagrangian $L(\bm{q},\dot{\bm{q}}) := T(\dot{\bm{q}}) - V(\bm{q})$, the kinetic energy $T(\dot{\bm{q}}) := \frac{1}{2}\langle\dot{\bm{q}},\dot{\bm{q}}\rangle$, and the potential energy $V(\bm{q})$ given by \eqref{defPotential}. The associated Hamiltonian $H(\bm{q},\dot{\bm{q}})$ defined as the Legendre-Fenchel conjugate \cite[p. 104]{rockafellar1970convex} of the Lagrangian $L$, equals $H(\bm{q},\dot{\bm{q}}) = \langle\frac{\partial L}{\partial\dot{\bm{q}}},\dot{\bm{q}}\rangle - L(\bm{q},\dot{\bm{q}}) = T(\dot{\bm{q}})+V(\bm{q})$, the total energy. Consequently, the equations of motion in Hamilton's canonical variables $(\bm{q},\bm{p}):=(\bm{q},\frac{\partial L}{\partial\dot{\bm{q}}})=(\bm{q},\dot{\bm{q}})$ becomes
$$\dot{\bm{q}}=\dfrac{\partial H}{\partial\bm{p}}, \quad \dot{\bm{p}}=-\dfrac{\partial H}{\partial\bm{q}},$$
which is identical to the dynamics  $\ddot{\bm{r}}=-\nabla_{\bm{q}}V(\bm{q})$ in \eqref{TPBVP}, as expected.

Replacing the endpoint constraints in \eqref{TPBVP} with \eqref{EndpointPDFs} is, therefore, equivalent to modifying the optimal control problem \eqref{LambertOCP} by replacing \eqref{LambertOCPterminalconstr} with \eqref{EndpointPDFs}. Furthermore, the probabilistic uncertainty in the initial condition $\bm{r}(t=t_0)\sim\rho_0$ is advected by the controlled sample path dynamics \eqref{LambertOCPdynconstr} resulting in the joint PDF evolution of the state  $\bm{r}(t)$ governed via the Liouville PDE initial value problem (see e.g., \cite{halder2010beyond,halder2011dispersion,haddad2022density}) 
\begin{align}
\dfrac{\partial\rho}{\partial t} + \nabla_{\bm{r}}\cdot \left(\rho\bm{v}\right)=0, \quad \rho(t=0,\cdot)=\rho_0\:\text{(given)},    
\label{LiouvilleIVP}    
\end{align}
where $\rho(\bm{r},t)$ denotes the \emph{controlled} transient joint state PDF for a given control policy $\bm{v}(\bm{r},t)$. The Liouville PDE $\dfrac{\partial\rho}{\partial t} + \nabla_{\bm{r}}\cdot \left(\rho\bm{v}\right)=0$ is the continuity equation signifying the conservation of probability mass over the state space, i.e., $\int_{\mathbb{R}^3} \rho(\bm{r},t)\:\differential\bm{r}=1$ for all $t\in[t_0,t_1]$. The solution of the Liouville PDE is to be understood in weak sense\footnote{This means that for all compactly supported smooth test functions $\phi(\bm{r},t)\in C_{c}^{
\infty}\left(\mathbb{R}^3\times [t_0,t_1]\right)$, the function $\rho(\bm{r},t)$ satisfies $\int_{t_0}^{t_1}\int_{\mathbb{R}^{3}}\left(\rho\frac{\partial\phi}{\partial t} + \rho\langle\bm{v},\nabla_{\bm{r}}\phi\rangle\right)\differential\bm{r}\differential t + \int_{\mathbb{R}^{3}}\rho_0(\bm{r})\phi(\bm{r},t=t_0)\differential\bm{r} = 0$.}.  

Different feasible control policies $\bm{v}$ in \eqref{LiouvilleIVP} induce different PDF-valued trajectories $\rho(\cdot,t)$ connecting the prescribed endpoint joints $\rho_0,\rho_1\in\mathcal{P}_2\left(\mathbb{R}^{3}\right)$. Hence, the objective in \eqref{LambertOCPobj} should be averaged w.r.t. the controlled transient joint probability measure $\rho(\bm{r},t)\:\differential\bm{r}$, and the resulting functional should be minimized over the pair $(\rho,\bm{v})$. We thus arrive at a \emph{generalized OMT} formulation
\begin{subequations}
\begin{align}
&\underset{\left(\rho,\bm{v}\right)\in\mathcal{P}_{01}\times\mathcal{V}}{\arg\inf}\quad\displaystyle\int_{t_0}^{t_1}\int_{\mathbb{R}^3}\left(\dfrac{1}{2}\vert \bm{v}\vert ^{2} - V(\bm{r})\right)\rho(\bm{r},t)\:\differential\bm{r}\:\differential t \label{LambertOTobj}\\  
&\dfrac{\partial\rho}{\partial t} + \nabla_{\bm{r}}\cdot \left(\rho\bm{v}\right)=0.\label{LambertOTdynconstr}
\end{align}
\label{LambertOT}
\end{subequations}
henceforth referred to as the \emph{Lambertian optimal mass transport (L-OMT)}.

Notice that in the special case when the endpoint joint PDFs are Dirac deltas at some given positions $\bm{r}_0,\bm{r}_1\in\mathbb{R}^3$, i.e., $\rho_0(\bm{r}) = \delta(\bm{r}-\bm{r}_0), \rho_1(\bm{r}) = \delta(\bm{r}-\bm{r}_1)$, then problem \eqref{LambertOT} reduces back to \eqref{LambertOCP}, or equivalently to \eqref{TPBVP}. 

What exactly makes problem \eqref{LambertOT} ``generalized''? The answer is the potential $V$, because if $V\equiv 0$ then \eqref{LambertOT} would be identical to the classical dynamic OMT \eqref{BenamouBrenierOMT}. From a control-theoretic perspective, the problem \eqref{LambertOT} is an atypical stochastic optimal control problem because it asks to minimize certain total expected ``cost-to-go'' from a given PDF to another under deadline and controlled dynamics constraints. We next establish
the existence-uniqueness of solution for \eqref{LambertOT}.

\begin{remark}
We emphasize here that the problem \eqref{ProbLambertFeasibility} with second order dynamics constraint is exactly equivalent to the L-OMT \eqref{LambertOT} with first order dynamics constraint: the nonlinearity in $V$ in the second order dynamics in the former has now been ``pushed" to the Lagrangian of the latter. This makes it possible to study a feasibility problem \eqref{ProbLambertFeasibility} in the manifold of PDF-valued curves as an equivalent optimization problem \eqref{LambertOT} in that manifold.
\end{remark}


\subsection{Existence and uniqueness of Solution}\label{subsec:LOMTexistenceunqueness} We start with a definition.
\begin{definition} [Superlinear function]
A function $f\colon\mathbb{R}^{d}\mapsto \mathbb{R}$ is superlinear or 1-coercive, if
$$\displaystyle\lim_{\vert\bm{x}\vert \rightarrow\infty}\dfrac{f(\bm{x})}{\vert\bm{x}\vert} = +\infty.$$
\label{defSuperlinear}
\end{definition}
We now state and prove the following for problem \eqref{LambertOT}. 

\begin{theorem}[Existence-uniqueness of solution for L-OMT]\label{thm:existenceuniqueness} 
Let $\rho_0,\rho_1\in\mathcal{P}_{2}(\mathbb{R}^{3})$, 
and the gravitational potential $V$ is negative and lower bounded as in \eqref{defPotential}. Then the minimizing tuple $\left(\rho^{\rm{opt}},\bm{v}^{\rm{opt}}\right)$ for problem \eqref{LambertOT} exists and is unique.
\end{theorem}

\begin{proof} The ``cost-to-go'' in \eqref{LambertOTobj} is the average of the Lagrangian
\begin{align}
    L(t, \bm{r}, \bm{\dot{r}}) = \frac{1}{2} \vert \bm{\dot{r}}\vert^2 - V(\bm{r}).
    \label{LOMTlagrangian}
\end{align}
Since $\frac{1}{2} \vert \cdot\vert^2$ is a strictly convex function, so $L$ is strictly convex in $\bm{\dot{r}}$. We now show that $L$ is also superlinear in $\bm{\dot{r}}$. Notice that
\begin{align}
    \lim_{\vert \bm{\dot{r}}\vert  \rightarrow \infty} \frac{L}{\vert \bm{\dot{r}}\vert } = \lim_{\vert \bm{\dot{r}}\vert \rightarrow \infty} \left( \frac{\frac{1}{2} \vert \bm{\dot{r}}\vert^2 - V(\bm{r})}{\vert \bm{\dot{r}}\vert } \right) = \left( \lim_{\vert\bm{\dot{r}}\vert \rightarrow \infty} \frac{1}{2}\vert \bm{\dot{r}}\vert \right) +  \left( \lim_{\vert\bm{\dot{r}}\vert \rightarrow \infty} \frac{-V(\bm{r})}{\vert\bm{\dot{r}}\vert} \right).
    \label{ShowingCoercive}
\end{align}
The first term $\lim_{\vert\bm{\dot{r}}\vert \rightarrow \infty} \frac{1}{2}\vert\bm{\dot{r}}\vert = +\infty$. For the second term in \eqref{ShowingCoercive}, since $V$ being negative and lower bounded, $-V$ is positive and upper bounded. Thus $\displaystyle\lim_{\vert\bm{\dot{r}}\vert \rightarrow \infty} \frac{-V(\bm{r})}{\vert\bm{\dot{r}}\vert} = 0$. Therefore, 
$$\lim_{\vert \bm{\dot{r}}\vert \rightarrow \infty} \frac{L}{\vert\bm{\dot{r}}\vert } = +\infty.$$
Being strictly convex and superlinear in $\bm{\dot{r}}$, the $L$ in \cref{LOMTlagrangian} is a weak Tonelli Lagrangian \cite[p. 118]{villani2009optimal}, \cite[Ch. 6.2]{figalli2007optimal}, which in turn guarantees \cite[Thm. 1.4.2]{figalli2007optimal} that the pair $\left(\rho^{\rm{opt}},\bm{v}^{\rm{opt}}\right)$ for the generalized OMT problem \cref{LambertOT} exists and is unique.
\end{proof}

\begin{remark}\label{RemarkVnegative}
The proof above only used that the potential $V$ is negative and lower bounded; it did not use the specific form \cref{defPotential}. Thus, the existence-uniqueness result above applies for any negative and lower bounded continuously differentiable gravitational potentials. To explicitly see why the $V$ in \eqref{defPotential} is negative, rewrite it as \begin{align}
V(\bm{r}) = -\dfrac{\mu}{|\bm{r}|} \left( 1 + \dfrac{J_2 R_{\rm{Earth}}^2}{2 |\bm{r}|^{2}} \left(1-\dfrac{3z^2}{|\bm{r}|^2} \right) \right). 
\label{vr}
\end{align}
Since $0 \leq z^2 \leq \vert\bm{r}\vert^2$, we have $-2 \leq 1 - \frac{3z^2}{\vert\bm{r}\vert^2} \leq 1$. On the other hand, $\vert\bm{r}\vert \geq R_{\rm{Earth}}$, which yields 
$$\dfrac{J_2 R_{\rm{Earth}}^2}{2 |\bm{r}|^{2}} \leq \dfrac{J_2 R_{\rm{Earth}}^2}{2 R_{\rm{Earth}}^{2}} \leq \dfrac{J_2}{2} < 0.0006.$$
Hence,
\begin{align}
    (-2)(0.0006) = -0.0012 < \dfrac{J_2 R_{\rm{Earth}}^2}{2 \vert\bm{r}\vert^{2}} \left(1-\dfrac{3z^2}{\vert \bm{r}\vert^2} \right). 
    \label{VrInequality}
\end{align}
Combining \eqref{vr} and 
\eqref{VrInequality}, we note that $V(\bm{r}) < 0$.
\end{remark}
We next further generalize the L-OMT \eqref{LambertOT} to allow for stochastic process noise. The resulting problem, as we clarify in \cref{sec:LambertSBP}, is a generalized version of the SBP \eqref{SBP}. Our motivation behind pursuing this generalization is twofold. \emph{First}, the conditions for optimality that we derive in \cref{sec:LambertSBP} helps design nonparametric numerical algorithms, and thereby provably solve the probabilistic Lambert problem in the ``small noise'' regime. From this perspective, the process noise plays the role of computational regularization. \emph{Second}, the same theory and algorithm apply when we indeed have (not necessarily small) dynamic process noise due to imperfect actuation, stochastic disturbance in atmospheric drag etc. as mentioned before in \cref{sec:Intro}.


\section{Stochastic regularization: Lambertian Schr\"{o}dinger Bridge}\label{sec:LambertSBP}

\subsection{Formulation and existence-uniqueness of Solution}\label{subsec:LSBPformulationexistenceuniqueness}
For not necessarily small $\varepsilon>0$, and given time horizon $[t_0,t_1]$ as before, we consider a stochastic regularized version of the L-OMT \eqref{LambertOT}: 
\begin{subequations}
\begin{align}
&\underset{\left(\rho,\bm{v}\right)\in\mathcal{P}_{01}\times\mathcal{V}}{\arg\inf}\quad\displaystyle\int_{t_0}^{t_1}\int_{\mathbb{R}^3}\left(\dfrac{1}{2}\vert \bm{v}\vert^{2} - V(\bm{r})\right)\rho(\bm{r},t)\:\differential\bm{r}\:\differential t \label{LambertSBPobj}\\  
&\dfrac{\partial\rho}{\partial t} + \nabla_{\bm{r}}\cdot \left(\rho\bm{v}\right)=\varepsilon\Delta_{\bm{r}}\rho. \label{LambertSBPdynconstr}
\end{align}
\label{LambertSBP}
\end{subequations}
which we refer to as the \emph{Lambertian Schr\"{o}dinger Bridge Problem (L-SBP)}. 

The dynamical constraint in \eqref{LambertSBP} differs from \eqref{LambertOT} by a scaled Laplacian term in the right-hand-side of \eqref{LambertSBPdynconstr}. The L-SBP \eqref{LambertSBP} generalizes the L-OMT \eqref{LambertOT} in a similar way the classical SBP \eqref{SBP} generalizes the classical OMT \eqref{BenamouBrenierOMT}.

For arbitrary $\varepsilon>0$, denote the minimizing pair for problem \eqref{LambertSBP} as $\left(\rho_{\varepsilon}^{\rm{opt}},\bm{v}_{\varepsilon}^{\rm{opt}}\right)$ wherein the subscript $\varepsilon$ emphasizes the solution's dependence on the stochastic regularization parameter $\varepsilon$. We next define the \emph{relative entropy} a.k.a. \emph{Kullback-Leibler divergence}, which plays a key role to establish existence-uniqueness for the pair $\left(\rho_{\varepsilon}^{\rm{opt}},\bm{v}_{\varepsilon}^{\rm{opt}}\right)$.

\begin{definition}\label{defKL}
(\textbf{Relative entropy a.k.a. Kullback-Leibler divergence}) Given two probability measures $\mathbb{P},\mathbb{Q}$ on some measure space $\Omega$, the relative entropy or Kullback-Leibler divergence 
\begin{align}
D_{\rm{KL}}\left(\mathbb{P}\parallel \mathbb{Q}\right) := \mathbb{E}_{\mathbb{P}}\left[\log\dfrac{\differential\mathbb{P}}{\differential\mathbb{Q}}\right] = \begin{cases}
\displaystyle\int_{\Omega}\log\left(\dfrac{\differential\mathbb{P}}{\differential\mathbb{Q}}\right)\differential\mathbb{P} & \text{if}\quad\mathbb{P} \ll \mathbb{Q},\\
+\infty & \text{otherwise,}
\end{cases}
\end{align}
where $\mathbb{E}_{\mathbb{P}}\left[\cdot\right]$ denotes the expectation operator w.r.t. the probability measure $\mathbb{P}$, the symbol $\dfrac{\differential\mathbb{P}}{\differential\mathbb{Q}}$ denotes the Radon-Nikodym derivative, and $\mathbb{P} \ll \mathbb{Q}$ means ``$\mathbb{P}$ is absolutely continuous w.r.t. $\mathbb{Q}$''.
\end{definition}

\begin{theorem}\label{ThmLSBPExistenceUniqueness} (\textbf{Existence-uniqueness of solution for L-SBP})
Let $\rho_0,\rho_1\in\mathcal{P}_2\left(\mathbb{R}^{3}\right)$, $\varepsilon>0$, and the gravitational potential $V(\cdot)\in(-\infty,0)$ as in \eqref{defPotential}. Then the minimizing tuple $\left(\rho_{\varepsilon}^{\rm{opt}},\bm{v}_{\varepsilon}^{\rm{opt}}\right)$ for problem \eqref{LambertSBP} exists and is unique.
\end{theorem}
\begin{proof}
The main idea behind this proof is to recast \eqref{LambertSBP} as a relative entropy minimization problem w.r.t. a reference Gibbs measure arising from the potential $V$. Here, the role of the reference measure will be played by the Wiener measure on the path space $\Omega:=C\left([t_0,t_1];\mathbb{R}^{3}\right)$, i.e., the space of continuous curves in $\mathbb{R}^{3}$ parameterized by $t\in[t_0,t_1]$. To this end, we follow the developments as in \cite{wakolbinger1990schrodinger} and \cite{dawson1990schrodinger}; see also \cite{leonard2011stochastic}. We give a self-contained proof here 
because \cite{wakolbinger1990schrodinger} and \cite{dawson1990schrodinger} may not be not readily accessible, and also because as written for probabilists, both omit certain technical details. The latter may otherwise pose challenges to theoretically minded engineers--our intended audience--in adapting the core ideas for the case in hand.  

Let $\mathcal{M}\left(\Omega\right)$ be the collection of all probability measures on the path space $\Omega$. Recall that \eqref{LambertSBPdynconstr} corresponds to the It\^{o} diffusion process \eqref{SDESamplePath} where $\bm{v}$ is non-anticipating, i.e., adapted to the filtration generated up until time $t$ for all $t\in[t_0,t_1]$. Let $\mathbb{W}\in\mathcal{M}\left(\Omega\right)$ be a Wiener measure, and consider a measure $\mathbb{P}\in\mathcal{M}\left(\Omega\right)$ generated by the It\^{o} diffusion \eqref{SDESamplePath}. Using Girsanov's theorem \cite[Thm. 8.6.6]{oksendal2013stochastic}, \cite[Ch. 3.5]{karatzas2012brownian}, we find
\begin{subequations}
\begin{align}
\dfrac{\differential\mathbb{P}}{\differential \mathbb{W}} &= \dfrac{\differential\mathbb{P}_0}{\differential \mathbb{W}_0}\exp\left(\int_{t_0}^{t_1}\dfrac{\bm{v}}{\sqrt{2\varepsilon}} \differential\bm{w}(t) +\int_{t_0}^{t_1}\dfrac{1}{2}\dfrac{\vert \bm{v} \vert^2}{2\varepsilon}\differential t\right), \qquad \mathbb{P} \;\: \text{a.s.},\\
\Leftrightarrow\quad\log\dfrac{\differential\mathbb{P}}{\differential \mathbb{W}} &= \log\dfrac{\differential\mathbb{P}_0}{\differential \mathbb{W}_0} + \dfrac{1}{\sqrt{2\varepsilon}}\int_{t_0}^{t_1}\bm{v} \differential\bm{w}(t) + \dfrac{1}{4\varepsilon}\int_{t_0}^{t_1}\vert \bm{v}\vert^2 \differential t, \qquad \mathbb{P} \;\: \text{a.s.},
\end{align}
\label{ApplyGirsanov}    
\end{subequations}
where $\mathbb{P}_0,\mathbb{W}_0$ are the distributions of $\bm{r}(t=t_0)$ under $\mathbb{P}$ and $\mathbb{W}$, respectively. 

Since $\bm{v}\in\mathcal{V}$ as defined in \eqref{DefFerasibleControl}, we have
\begin{align}
\mathbb{E}_{\mathbb{P}}\left[\displaystyle\int_{t_0}^{t_1}\vert\bm{v}\vert^2\differential t\right] < \infty.
\end{align}
So for any $t\geq t_0$, the stochastic process 
$\overline{\bm{v}}(t) := \int_{t_0}^{t}\bm{v} \differential\bm{w}(\tau)$ is a martingale, and consequently has constant expected value $\mathbb{E}_{\mathbb{P}}\left[\overline{\bm{v}}(t)\right] = \mathbb{E}_{\mathbb{P}}\left[\overline{\bm{v}}(t_0)\right] = 0$. Therefore, applying the expectation operator w.r.t. $\mathbb{P}$ to both sides of \eqref{ApplyGirsanov}, we obtain
\begin{align}
\mathbb{E}_{\mathbb{P}}\left[\log\dfrac{\differential\mathbb{P}}{\differential \mathbb{Q}}\right]
= \mathbb{E}_{\mathbb{P}}\left[\log\dfrac{\differential\mathbb{P}_0}{\differential \mathbb{W}_0}\right]  + \dfrac{1}{4\varepsilon}\int_{t_0}^{t_1}\mathbb{E}_{\mathbb{P}}\left[\vert \bm{v}\vert^2\right]\differential t.
\label{ExpectedLogRadonNikodym}
\end{align}

Recalling that the potential $V(\cdot)$ is negative and lower bounded (see Remark \ref{RemarkVnegative}), we next consider the Gibbs measure
$$\dfrac{\exp\left(\dfrac{1}{2\varepsilon}\displaystyle\int_{t_0}^{t_1}V(\bm{r})\differential t\right)\mathbb{W}}{Z}\in\mathcal{M}\left(\Omega\right), \; \text{where}\;Z := \int_{\Omega}\exp\left(\dfrac{1}{2\varepsilon}\displaystyle\int_{t_0}^{t_1}V(\bm{r})\differential t\right)\differential\mathbb{W}.$$
By Definition \ref{defKL}, we get
\begin{align}
&D_{\rm{KL}}\left(\mathbb{P}\parallel \dfrac{\exp\left(\dfrac{1}{2\varepsilon}\displaystyle\int_{t_0}^{t_1}V(\bm{r})\differential t\right)\mathbb{W}}{Z}\right)\nonumber\\
&= \log Z + \mathbb{E}_{\mathbb{P}}\left[-\dfrac{1}{2\varepsilon}\displaystyle\int_{t_0}^{t_1}V(\bm{r})\differential t\right] + \mathbb{E}_{\mathbb{P}}\left[\log\dfrac{\differential\mathbb{P}}{\differential\mathbb{W}}\right] \nonumber\\
&= \log Z + \mathbb{E}_{\mathbb{P}}\left[\log\dfrac{\differential\mathbb{P}_0}{\differential \mathbb{W}_0}\right]  + \dfrac{1}{2\varepsilon}\int_{t_0}^{t_1}\mathbb{E}_{\mathbb{P}}\left[\dfrac{1}{2}\vert \bm{v}\vert^2 -V(\bm{r})\right]\differential t,
\label{RelativeEntropy}
\end{align}
where the last equality uses \eqref{ExpectedLogRadonNikodym}. Now, let 
\begin{align}
\Pi_{01}:=\big\{\mathbb{M}\in\mathcal{M}\left(\Omega\right) \mid\mathbb{M}\;\text{has marginal }\:\rho_{i}\differential\bm{r}\;\text{at time}\;t_i \;\forall i\in\{0,1\},\:\rho_0,\rho_1\in\mathcal{P}_2\left(\mathbb{R}^{3}\right)\big\}.
\label{defPi01}    
\end{align}
Since $\log Z + \mathbb{E}_{\mathbb{P}}\left[\log\dfrac{\differential\mathbb{P}_0}{\differential \mathbb{W}_0}\right]$ is constant over $\Pi_{01}$, we conclude from \eqref{RelativeEntropy} that the stochastic optimal control problem \eqref{LambertSBP} is equivalent to the stochastic calculus of variations problem:
\begin{align}
\underset{\mathbb{P}\in\Pi_{01}}{\arg\inf}\quad D_{\rm{KL}}\left(\mathbb{P}\parallel \dfrac{\exp\left(\dfrac{1}{2\varepsilon}\displaystyle\int_{t_0}^{t_1}V(\bm{r})\differential t\right)\mathbb{W}}{Z}\right).
\label{LSBPasKLminimization}
\end{align}
From \eqref{defPi01}, the set $\Pi_{01}$ is closed and convex with non-empty interior. Furthermore, the mapping $\mathbb{P} \mapsto D_{\rm{KL}}\left(\mathbb{P}\parallel \mathbb{Q}\right)$ is strictly convex over $\Pi_{01}$ for a fixed measure $\mathbb{Q}$. Hence there exists unique solution for \eqref{LSBPasKLminimization}, or equivalently for \eqref{LambertSBP}. 
\end{proof}

\begin{remark}\label{remarkSanovThmLargeDevPrinciple}
Thanks to the conditional Sanov's theorem \cite[Sec. 4]{wakolbinger1990schrodinger}, \cite{csiszar1984sanov}, \eqref{LSBPasKLminimization} also shows that the problem \eqref{LambertSBP} enjoys a large deviations principle with rate function $D_{\rm{KL}}\left(\cdot\parallel \frac{\exp\left(\dfrac{1}{2\varepsilon}\int_{t_0}^{t_1}V(\bm{r})\differential t\right)\mathbb{W}}{Z}\right)$. This can be understood intuitively as follows. For $n$ i.i.d. random paths in $\Omega$ with distribution $\mathbb{W}$, when $n$ is large, the more likely paths in $\Pi_{01}$ correspond to those $\mathbb{P}$ which make the objective in \eqref{LSBPasKLminimization} small. In particular, for $n\rightarrow\infty$, the path likelihood is asymptotically equivalent to $$\exp\left(-n \underset{\mathbb{P}\in\Pi_{01}}{\inf}\:D_{\rm{KL}}\left(\mathbb{P}\parallel \frac{\exp\left(\dfrac{1}{2\varepsilon}\displaystyle\int_{t_0}^{t_1}V(\bm{r})\differential t\right)\mathbb{W}}{Z}\right)\right).$$
The most likely path in $\Pi_{01}$ corresponds to the minimizer of \eqref{LSBPasKLminimization}.
\end{remark}

As $\varepsilon \downarrow 0$, clearly the L-SBP \eqref{LambertSBP} reduces to the L-OMT \eqref{LambertOT}. More importantly, as $\varepsilon \downarrow 0$, the solution to \eqref{LambertSBP} converges \cite{leonard2012schrodinger} to the solution of \eqref{LambertOT}, i.e., $$\left(\rho_{\varepsilon}^{\rm{opt}}, \bm{v}_{\varepsilon}^{\rm{opt}}\right)\xrightarrow{\varepsilon\downarrow 0} \left(\rho^{\rm{opt}},\bm{v}^{\rm{opt}}\right).$$
In the following, we derive the system of equations that solves the optimal pair $\left(\rho_{\varepsilon}^{\rm{opt}}, \bm{v}_{\varepsilon}^{\rm{opt}}\right)$.

\subsection{Conditions for optimality}\label{subsec:LSBPoptimality}
The Proposition \ref{PropCondOptimalityLSBP} next shows that the first order necessary conditions for optimality for the L-SBP \eqref{LambertSBP} takes the form of a coupled system of nonlinear PDEs.
\begin{proposition}[Conditions for optimality for L-SBP]\label{PropCondOptimalityLSBP} The pair $\left(\rho_{\varepsilon}^{\rm{opt}},\bm{v}_{\varepsilon}^{\rm{opt}}\right)$ solving the L-SBP \eqref{LambertSBP}, satisfies the system of coupled nonlinear PDEs
\begin{subequations}
    \begin{align}
        &\dfrac{\partial\psi_{\varepsilon}}{\partial t} +\dfrac{1}{2}\vert \nabla_{\bm{r}} \psi_{\varepsilon} \vert^{2} +  \varepsilon\Delta_{\bm{r}}\psi_{\varepsilon} = - V(\bm{r}),\label{HJBSBP}\\
         &\dfrac{\partial\rho_{\varepsilon}^{\rm{opt}}}{\partial t} + \nabla_{\bm{r}}\cdot \left(\rho_{\varepsilon}^{\rm{opt}}\nabla_{\bm{r}} \psi_{\varepsilon}\right)=\varepsilon\Delta_{\bm{r}}\rho_{\varepsilon}^{\rm{opt}}, \label{FPKSBP}
    \end{align}
    \label{FirstOrderConditions}
\end{subequations}
with boundary conditions
\begin{align}
    \rho_{\varepsilon}^{\rm{opt}}(\bm{r},t=t_0) = \rho_0(\bm{r}), \quad \rho_{\varepsilon}^{\rm{opt}}(\bm{r},t=t_1) = \rho_1(\bm{r}),
    \label{BoundaryConditionsForProp1}
\end{align}
where $\psi_{\varepsilon}\in C^{1,2}\left(\mathbb{R}^{3};[t_0,t_1]\right)$, and the optimal control 
\begin{align}
    \bm{v}_{\varepsilon}^{\rm{opt}} = \nabla_{\bm{r}} \psi_{\varepsilon}(\bm{r}, t).
    \label{inputSBP}
\end{align}
\end{proposition}
\begin{proof}
The Lagrangian for the L-SBP \eqref{LambertSBP} is the functional 
\begin{align}
\int_{t_0}^{t_1}\!\!\!\int_{\mathbb{R}^3}\!\!\left(\!\left(\frac{1}{2} \vert \bm{v}\vert^2 - V(\bm{r}) \!\right) \rho(\bm{r},t) + \psi_{\varepsilon}(\bm{r}, t) \times \left(\!\frac{\partial \rho}{\partial t} + \nabla_{\bm{r}} \cdot (\rho \bm{v}) - \varepsilon \Delta_{\bm{r}} \rho \!\right)\!\right)\differential\bm{r}\:\differential t,
\label{LSBPLagrangian}
\end{align}
where the Lagrange multiplier $\psi_{\varepsilon}\in C^{1,2}\left(\mathbb{R}^{3};[t_0,t_1]\right)$. The functional \eqref{LSBPLagrangian} depends on the primal-dual variable trio $(\rho,\bm{v},\psi_{\varepsilon})$. The idea now is to perform unconstrained minimization of \eqref{LSBPLagrangian} over the feasible space $\mathcal{P}_{01}\times\mathcal{V}$. 

For the term $\int_{t_0}^{t_1} \int_{\mathbb{R}^3} \psi_{\varepsilon}(\bm{r}, t) \frac{\partial \rho}{\partial t}\:\differential\bm{r}\:\differential t$ in \eqref{LSBPLagrangian}, we apply Fubini-Tonelli theorem to switch the order of the integrals, and perform integration by parts w.r.t. $t$, to obtain
\begin{align}
    \int_{t_0}^{t_1} \!\!\!\int_{\mathbb{R}^3} \psi_{\varepsilon}(\bm{r}, t) \frac{\partial \rho}{\partial t}\:\differential\bm{r}\:\differential t = &\underbrace{\int_{\mathbb{R}^3} \!\!\left(\psi_{\varepsilon}(\bm{r},t_1)\rho_1(\bm{r})\nonumber-\psi_{\varepsilon}(\bm{r},t_0)\rho_0(\bm{r})\right)\differential\bm{r}}_{\text{constant over } \mathcal{P}_{01}\times\mathcal{V}}\\ &\qquad\qquad\qquad\qquad\qquad\qquad- \int_{t_0}^{t_1} \!\!\!\int_{\mathbb{R}^3}\rho(\bm{r}, t) \frac{\partial \psi_{\varepsilon}}{\partial t}\:\differential\bm{r}\:\differential t.
    \label{IntByParts}
\end{align}
Next, for the term $\int_{t_0}^{t_1} \int_{\mathbb{R}^3}\psi_{\varepsilon}(\bm{r},t)\left(\nabla_{\bm{r}} \cdot (\rho \bm{v}) - \varepsilon \Delta_{\bm{r}} \rho \right)\differential\bm{r}\:\differential t$ in \eqref{LSBPLagrangian}, integration by parts w.r.t. $\bm{r}$ while assuming the limits at $|\bm{r}|\rightarrow\infty$ are zero, gives
\begin{align}
\int_{t_0}^{t_1} \int_{\mathbb{R}^3}\psi_{\varepsilon}(\bm{r},t)\left(\nabla_{\bm{r}} \cdot (\rho \bm{v}) - \varepsilon \Delta_{\bm{r}} \rho(\bm{r}, t) \right)\differential\bm{r}\:\differential t &=-\int_{t_0}^{t_1} \int_{\mathbb{R}^3}\langle\nabla_{\bm{r}}\psi_{\varepsilon}, \bm{v}\rangle\rho\differential\bm{r}\differential t\nonumber\\ 
&\qquad - \varepsilon\int_{t_0}^{t_1} \int_{\mathbb{R}^3}\langle\nabla_{\bm{r}}\rho,\nabla_{\bm{r}}\psi_{\varepsilon}\rangle\differential\bm{r}\differential t. 
\label{IBPr}    
\end{align}
For the second integral above, performing integration by parts w.r.t. $\bm{r}$ once more, \eqref{IBPr} becomes
\begin{align}
\int_{t_0}^{t_1} \int_{\mathbb{R}^3}\psi_{\varepsilon}(\bm{r},t)\left(\nabla_{\bm{r}} \cdot (\rho \bm{v}) - \varepsilon \Delta_{\bm{r}} \rho(\bm{r}, t) \right)\differential\bm{r}\:\differential t = &-\int_{t_0}^{t_1} \int_{\mathbb{R}^3}\langle\nabla_{\bm{r}}\psi_{\varepsilon}, \bm{v}\rangle\rho(\bm{r},t)\differential\bm{r}\differential t\nonumber\\
&+ \varepsilon\int_{t_0}^{t_1} \rho(\bm{r},t)\Delta_{\bm{r}}\psi_{\varepsilon}\differential\bm{r}\:\differential t.
\label{IntByPartsTwoFold}
\end{align}
Substituting \eqref{IntByParts} and \eqref{IntByPartsTwoFold} back in \eqref{LSBPLagrangian}, and dropping the constant term, we arrive at the expression
\begin{align}
\int_{t_0}^{t_1}\int_{\mathbb{R}^3}\left(\dfrac{1}{2}\vert\bm{v}\vert^{2} - V(\bm{r}) - \dfrac{\partial\psi_{\varepsilon}}{\partial t} -  \langle\nabla_{\bm{r}}\psi_{\varepsilon},\bm{v}\rangle - \varepsilon\Delta_{\bm{r}}\psi_{\varepsilon} \right) \rho(\bm{r}, t) \:\differential\bm{r}\:\differential t.
\label{NewLSBPLagrangian}
\end{align}
Pointwise minimization of \eqref{NewLSBPLagrangian} w.r.t. $\bm{v}$ for a fixed $\rho\in\mathcal{P}_{01}$, gives optimal control
\begin{align*}
    \bm{v}_{\varepsilon}^{\rm{opt}} = \nabla_{\bm{r}} \psi_{\varepsilon}(\bm{r}, t),
\end{align*}
which is precisely \eqref{inputSBP}. Evaluating \eqref{NewLSBPLagrangian} at this optimal solution, and equating the resulting expression to zero yields the dynamic programming equation
\begin{align}
    0 =\int_{t_0}^{t_1}\int_{\mathbb{R}^3}\left(- \dfrac{1}{2}\vert\nabla_{\bm{r}} \psi_{\varepsilon} \vert^{2} - V(\bm{r}) - \dfrac{\partial\psi_{\varepsilon}}{\partial t} -  \varepsilon\Delta_{\bm{r}}\psi_{\varepsilon} \right) \rho_{\varepsilon}^{\rm{opt}}(\bm{r}, t) \:\differential\bm{r}\:\differential t.
    \label{HJBbeforefinal}
\end{align}
For \eqref{HJBbeforefinal} to hold for arbitrary $\rho\in\mathcal{P}_{01}$, we must have
\begin{align}
    \dfrac{1}{2}\vert\nabla_{\bm{r}} \psi_{\varepsilon} \vert^{2} + \dfrac{\partial\psi_{\varepsilon}}{\partial t} +  \varepsilon\Delta_{\bm{r}}\psi_{\varepsilon} = - V(\bm{r}),
\end{align}
which is the PDE \eqref{HJBSBP}.

Substituting \eqref{inputSBP} into the primal feasibility constraint \eqref{LambertSBPdynconstr} yields \eqref{FPKSBP}. Finally, the boundary conditions \eqref{BoundaryConditionsForProp1} follow from the given endpoint constraints encoded in $\mathcal{P}_{01}$.
\end{proof}

For $\varepsilon=0$, the proof above goes through, mutatis mutandis, giving the following for the L-OMT solution $(\rho^{\text{opt}},\bm{v}^{\text{opt}})$ mentioned at the end of \cref{subsec:LSBPformulationexistenceuniqueness}.
\begin{corollary}[Conditions for optimality for L-OMT]\label{CorollaryLOMTconditionsOptimality} 
The pair $(\rho^{\rm{opt}},\bm{v}^{\rm{opt}})$ solving the L-OMT \eqref{LambertOT}, satisfies the system of coupled nonlinear PDEs
\begin{subequations}
    \begin{align}
        &\dfrac{\partial\psi}{\partial t} +\dfrac{1}{2}\vert \nabla_{\bm{r}} \psi \vert^{2} + V(\bm{r}) = 0,\label{HJBOMT}\\
         &\dfrac{\partial\rho^{\rm{opt}}}{\partial t} + \nabla_{\bm{r}}\cdot \left(\rho^{\rm{opt}}\nabla_{\bm{r}} \psi\right)=0, \label{LiouvilleOMT}
    \end{align}
    \label{FirstOrderConditionsLMOT}
\end{subequations}
with boundary conditions $\rho^{\rm{opt}}(\bm{r},t=t_0) = \rho_0(\bm{r}), \rho^{\rm{opt}}(\bm{r},t=t_1) = \rho_1(\bm{r})$, where $\psi\in C^{1,1}\left(\mathbb{R}^{3};[t_0,t_1]\right)$, and the optimal control 
$\bm{v}^{\rm{opt}} = \nabla_{\bm{r}} \psi(\bm{r}, t)$.
\end{corollary}

\begin{remark}
The second order PDE \eqref{HJBSBP} is a Hamilton-Jacobi-Bellman (HJB) equation, whereas the second order PDE \eqref{FPKSBP} is a controlled Fokker-Planck-Kolmogorov (FPK) equation. At the equation level, these two are coupled one way: \eqref{HJBSBP} only depends on the dual variable $\psi_{\varepsilon}$ while \eqref{FPKSBP} depends on both $\psi_{\varepsilon}$ and $\rho_{\varepsilon}^{\rm{opt}}$. The boundary conditions \eqref{BoundaryConditionsForProp1}, however, are only available for the endpoints of the primal variable $\rho_{\varepsilon}^{\rm{opt}}$. 
\end{remark}

We next focus on finding a solution strategy for the system of coupled \emph{nonlinear} PDEs \eqref{FirstOrderConditions} with unconventional boundary conditions \eqref{BoundaryConditionsForProp1}. In the following Theorem \ref{ThmHopfColeReactionDiffusion}, we use the Hopf-Cole transform \cite{hopf1950partial,cole1951quasi} to show that the system \eqref{FirstOrderConditions}-\eqref{BoundaryConditionsForProp1} can be exactly converted to a system of \emph{linear} reaction-diffusion PDEs at the expense of moving the coupling to the boundary conditions. In \cref{sec:Algorithm}, we explain how this transformed system is malleable for numerical computation.  

\begin{theorem} [Boundary-coupled system of linear reaction-diffusion PDEs] \label{ThmHopfColeReactionDiffusion}Consider the L-SBP \eqref{LambertSBP} with given $[t_0,t_1]$, $V$, $\varepsilon$, $\rho_0$, $\rho_1$ as in Theorem \ref{ThmLSBPExistenceUniqueness}. Let $\left(\rho^{\mathrm{opt}}_\varepsilon, \psi_{\varepsilon}\right)$ be the solution of \eqref{FirstOrderConditions}-\eqref{BoundaryConditionsForProp1}. Consider the Hopf-Cole transform $\left(\rho^{\mathrm{opt}}_\varepsilon, \psi_{\varepsilon}\right) \mapsto(\widehat{\varphi}_{\varepsilon},\varphi_{\varepsilon})$ defined as
\begin{subequations}
\begin{align}
\varphi_{\varepsilon} &= \exp \left( \frac{\psi_{\varepsilon}}{2 \varepsilon} \right), \label{phi_epsilon}\\
\widehat{\varphi}_{\varepsilon} &= \rho^{\mathrm{opt}}_\varepsilon \exp \left( - \frac{\psi_{\varepsilon}}{2 \varepsilon} \right).\label{phi_hat_epsilon}
\end{align} 
\label{defphiphihat}
\end{subequations}
Then, the pair $(\widehat{\varphi}_{\varepsilon},\varphi_{\varepsilon})$ called the Schr\"{o}dinger factors, solve the system of forward and backward linear reaction-diffusion PDEs:
\begin{subequations}
\begin{align}
\displaystyle\frac{\partial\widehat{\varphi}_{\varepsilon}}{\partial t} &= \left(\varepsilon\Delta_{\bm{r}} + \frac{1}{2\varepsilon} V(\bm{r})\right)\widehat{\varphi}_{\varepsilon},    
\label{FactorPDEForward}\\
\displaystyle\frac{\partial\varphi_{\varepsilon}}{\partial t} &= -\left(\varepsilon\Delta_{\bm{r}} + \frac{1}{2\varepsilon}V(\bm{r})\right)\varphi_{\varepsilon},
\label{FactorPDEBackward}
\end{align}
\label{FactorPDEs}
\end{subequations}
with coupled boundary conditions
\begin{align}
\widehat{\varphi}_{\varepsilon}(\cdot,t=t_0)\varphi_{\varepsilon}(\cdot,t=t_0) = \rho_0, \quad \widehat{\varphi}_{\varepsilon}(\cdot,t=t_1)\varphi_{\varepsilon}(\cdot,t=t_1) = \rho_1.
\label{factorBC}    
\end{align}
For all $t\in[t_0,t_1]$, the minimizing pair for \eqref{LambertSBP} can be recovered from the Schr\"{o}dinger factors $(\widehat{\varphi}_{\varepsilon},\varphi_{\varepsilon})$ as
\begin{subequations}
\begin{align}
\rho_{\varepsilon}^{\rm{opt}}(\cdot,t) &= \widehat{\varphi}_{\varepsilon}(\cdot,t)\varphi_{\varepsilon}(\cdot,t) \label{OptimallyControlledJointPDF},\\
\bm{v}_{\varepsilon}^{\rm{opt}}(t,\cdot) &= 2 \varepsilon\nabla_{(\cdot)}\log\varphi_{\varepsilon}(\cdot,t).\label{OptimalVelocity}
\end{align}
\label{RecoverOptimalSolution}
\end{subequations}
\end{theorem}
\begin{proof}
By definition, $\psi_{\varepsilon}\in C^{1,2}\left([t_0,t_1],\mathbb{R}^3\right)$. Being continuous over its entire domain, $\psi_{\varepsilon}$ is bounded. So the mapping \eqref{defphiphihat} is bijective. In particular, both $\widehat{\varphi}_{\varepsilon},\varphi_{\varepsilon}$ are positive over the support of $\rho_{\varepsilon}^{\rm{opt}}$ for all $t\in[t_0,t_1]$.

From \eqref{phi_epsilon}, \begin{align}
   \psi_{\varepsilon} = 2\varepsilon \log\varphi_\varepsilon.
   \label{psilogphi}
\end{align}
From \eqref{phi_hat_epsilon}, $
    \widehat{\varphi}_\varepsilon = \rho^{\mathrm{opt}}_\varepsilon \exp(-\log(\varphi_\varepsilon))$,
which immediately gives \eqref{OptimallyControlledJointPDF}. Substituting \eqref{psilogphi} in \eqref{inputSBP} gives \eqref{OptimalVelocity}. Since \eqref{OptimallyControlledJointPDF} holds for all $t\in[t_0,t_1]$, so combining \eqref{BoundaryConditionsForProp1} and \eqref{OptimallyControlledJointPDF} yields \eqref{factorBC}. All that remains is to derive \eqref{FactorPDEs}.

Substituting \eqref{psilogphi} in \eqref{HJBSBP}, we get
\begin{align}
2 \varepsilon^2 \bigg\vert \frac{\nabla_{\bm{r}}  \varphi_\varepsilon}{\varphi_\varepsilon} \bigg\vert^{2} + 2\varepsilon \frac{1}{\varphi_\varepsilon} \dfrac{\partial \varphi_\varepsilon}{\partial t} +  \varepsilon\Delta_{\bm{r}}(2\varepsilon \log\varphi_\varepsilon) = - V(\bm{r}).
\label{SubInHJB}
\end{align}
Since $\Delta_{\bm{r}}(2\varepsilon \log\varphi_\varepsilon) = 2\varepsilon \nabla_{\bm{r}} \nabla_{\bm{r}}\log\varphi_\varepsilon = 2\varepsilon \nabla_{\bm{r}} \left( \frac{\nabla_{\bm{r}}\varphi}{\varphi_\varepsilon} \right)= 2\varepsilon \left( - \frac{\langle \nabla_{\bm{r}}\varphi_\varepsilon, \nabla_{\bm{r}}\varphi_\varepsilon\rangle}{\varphi_\varepsilon^2} + \frac{\Delta_{\bm{r}} \varphi_\varepsilon}{\varphi_\varepsilon}  \right)$, \eqref{SubInHJB} simplifies to
\begin{align}
    &2 \varepsilon^2 \bigg\vert \frac{\nabla_{\bm{r}}  \varphi_\varepsilon}{\varphi_\varepsilon} \bigg\vert^{2} \!\!+ 2\varepsilon \frac{1}{\varphi_\varepsilon} \dfrac{\partial \varphi_\varepsilon}{\partial t} +   2\varepsilon^2 \left( - \frac{\langle\nabla_{\bm{r}}\varphi_\varepsilon, \nabla_{\bm{r}}\varphi_\varepsilon\rangle}{\varphi_\varepsilon^2} + \frac{\Delta_{\bm{r}} \varphi_\varepsilon}{\varphi_\varepsilon} \right) = - V(\bm{r}),\nonumber\\
\Leftrightarrow\quad
     &2\varepsilon \frac{1}{\varphi_\varepsilon} \dfrac{\partial \varphi_\varepsilon}{\partial t} +   2\varepsilon^2 \left( \frac{1}{\varphi_\varepsilon} \Delta_{\bm{r}} \varphi \right) = - V(\bm{r}).
     \label{phiPDEsimplified}
\end{align}
Rearranging \eqref{phiPDEsimplified}, we obtain \eqref{FactorPDEBackward}.

Next, substituting \eqref{OptimallyControlledJointPDF} and \eqref{psilogphi} in \eqref{FPKSBP}, we find
\begin{align}
    &\dfrac{\partial}{\partial t} \left( \widehat{\varphi}_\varepsilon \varphi \right) + \nabla_{\bm{r}}\cdot \left( \widehat{\varphi}_\varepsilon \varphi_\varepsilon \nabla_{\bm{r}}(2\varepsilon \log\varphi_\varepsilon) \right)=\varepsilon\Delta_{\bm{r}} \left( \widehat{\varphi}_\varepsilon \varphi_\varepsilon \right), \nonumber\\
    \Leftrightarrow\quad &\varphi_\varepsilon \dfrac{\partial \widehat{\varphi}_\varepsilon}{\partial t} + \widehat{\varphi}_\varepsilon \dfrac{\partial \varphi_\varepsilon}{\partial t} + \nabla_{\bm{r}}\cdot \left( 2\varepsilon \widehat{\varphi}_\varepsilon \varphi_\varepsilon \frac{\nabla_{\bm{r}}\varphi_\varepsilon}{\varphi_\varepsilon} \right)= \varepsilon \left( \varphi_\varepsilon \Delta_{\bm{r}} \widehat{\varphi}_\varepsilon + 2 \langle\nabla_{\bm{r}} \widehat{\varphi}_\varepsilon, \nabla_{\bm{r}}\varphi_\varepsilon\rangle + \widehat{\varphi}_\varepsilon \Delta_{\bm{r}} \varphi_\varepsilon  \right),
    \label{phihatPDEintermed}
\end{align}
where the right-hand-side of \eqref{phihatPDEintermed} used the identity $\Delta (fg) = f \Delta g + 2 \langle\nabla f,\nabla g\rangle + g \Delta f$ for twice differentiable $f,g$. 

We expand the term $\nabla_{\bm{r}}\cdot \left( 2\varepsilon \widehat{\varphi}_\varepsilon \varphi_\varepsilon \frac{1}{\varphi_\varepsilon} \nabla_{\bm{r}}\varphi_\varepsilon \right)$ appearing in the left-hand-side of \eqref{phihatPDEintermed} as
$$\nabla_{\bm{r}}\cdot \left( 2\varepsilon \widehat{\varphi}_\varepsilon \varphi_\varepsilon \frac{1}{\varphi_\varepsilon} \nabla_{\bm{r}}\varphi_\varepsilon \right) = 2\varepsilon \langle\nabla_{\bm{r}} \widehat{\varphi}_\varepsilon, \nabla_{\bm{r}}\varphi_\varepsilon\rangle + 2\varepsilon \widehat{\varphi}_\varepsilon \Delta_{\bm{r}}\varphi_\varepsilon,$$
and therefore, \eqref{phihatPDEintermed} simplifies to
\begin{align}
    &\varphi_\varepsilon \dfrac{\partial \widehat{\varphi}_\varepsilon}{\partial t} + \widehat{\varphi}_\varepsilon \dfrac{\partial \varphi_\varepsilon}{\partial t} +  \varepsilon \widehat{\varphi}_\varepsilon \Delta_{\bm{r}}\varphi_\varepsilon =\varepsilon \varphi_\varepsilon \Delta_{\bm{r}} \widehat{\varphi}_\varepsilon,\nonumber\\
    \Leftrightarrow\quad &\varphi_\varepsilon \dfrac{\partial \widehat{\varphi}_\varepsilon}{\partial t} + \widehat{\varphi}_\varepsilon\left(\dfrac{\partial \varphi_\varepsilon}{\partial t} + \varepsilon\Delta_{\bm{r}}\varphi_\varepsilon\right) = \varepsilon \varphi_\varepsilon \Delta_{\bm{r}} \widehat{\varphi}_\varepsilon,\nonumber\\
    \Leftrightarrow\quad &\varphi_\varepsilon \dfrac{\partial \widehat{\varphi}_\varepsilon}{\partial t} + \widehat{\varphi}_\varepsilon\left(-\frac{1}{2\varepsilon}V(\bm{r})\varphi_\varepsilon\right) = \varepsilon \varphi_\varepsilon \Delta_{\bm{r}} \widehat{\varphi}_\varepsilon,
\label{phihatPDEfinal}    
\end{align}
where the last line's left-hand-side used the already derived PDE \eqref{FactorPDEBackward}. Rearranging  \eqref{phihatPDEfinal}, we arrive at \eqref{FactorPDEForward}. This completes the proof.
\end{proof}

\begin{remark}
The derived \eqref{FactorPDEs}-\eqref{factorBC} is a boundary coupled system of linear reaction-diffusion PDEs. Specifically, \eqref{FactorPDEForward} (resp. \eqref{FactorPDEBackward}) is a forward-in-time (resp. backward-in-time) reaction-diffusion PDE with state-dependent reaction rate $V(\bm{r})$. Once this PDE system for the Schr\"{o}dinger factors is solved, the optimally controlled joint PDF and the optimal control can then be computed using \eqref{RecoverOptimalSolution}. In \cref{sec:Algorithm}, we will outline how to solve \eqref{FactorPDEs}-\eqref{factorBC}. Notice that the special case $V\equiv 0$ corresponding to the classical SBP in \cref{subsec:SBP}, reduces \eqref{FactorPDEs} to a system of forward-backward heat PDEs. We note that a special case of \eqref{FactorPDEs}-\eqref{factorBC} with Gaussian $\rho_0,\rho_1$ and convex quadratic $V$ was studied in \cite{7170905}. The case of generic $\rho_0,\rho_1$ with convex quadratic $V$ appeared recently in \cite{teter2024schr,teter2024weyl}.
\end{remark}

\begin{remark}
The derived system \eqref{FactorPDEs}-\eqref{factorBC} has the following probabilistic interpretation. This is a system of forward-backward diffusions where individual points in the state space $\mathbb{R}^3$ move forward and reverse in time according to Brownian motion with variance $2\varepsilon$, together with killing or creation of probability mass at rate $V(\bm{r})$. Thus, $V(\bm{r})$ in \eqref{FactorPDEForward} is position-dependent killing rate (recall $V$ is negative). Likewise, $V(\bm{r})$ in \eqref{FactorPDEBackward} is position-dependent creation rate. The boundary conditions \eqref{factorBC} ensure that the total probability mass remains balanced. For gravitational potentials such as \eqref{defPotential}, the killing and creation rates are small (resp. large) when $|\bm{r}|$ is large (resp. small).
\end{remark}


\section{Algorithm}\label{sec:Algorithm}
We start by outlining our overall approach (\cref{subsec:OverallApproach}) to numerically solve the system \eqref{FactorPDEs}-\eqref{factorBC}. We then provide some details (\cref{subsec:IntegralRepresentation}-\cref{subsec:Riemann}) needed to carry out that overall approach.

\subsection{Overall Approach}\label{subsec:OverallApproach}
Our high level idea is to solve \eqref{FactorPDEs}-\eqref{factorBC} via a recursive algorithm. Specifically, suppose that we have access to initial value problem (IVP) solvers for the PDEs \eqref{FactorPDEs}. However, the endpoint Schr\"{o}dinger factors 
\begin{align}
\widehat{\varphi}_{\varepsilon,0}(\cdot) := \widehat{\varphi}_{\varepsilon}(\cdot,t=t_0), \quad \varphi_{\varepsilon,1}(\cdot) := \varphi_{\varepsilon}(\cdot,t=t_1),
\label{defEndpointFactors}
\end{align}
in \eqref{factorBC} are not known. The proposed recursive algorithm starts with making an initial (everywhere positive) guess for the endpoint factor $\widehat{\varphi}_{\varepsilon,0}$, and with this initial guess, we integrate \eqref{FactorPDEForward} forward in time using the IVP solver to predict $\widehat{\varphi}_{\varepsilon}(\cdot,t=t_1)$. Then, applying the known boundary condition \eqref{factorBC} at $t=t_1$ generates a guess for $\varphi_{\varepsilon,1}(\cdot)$, which we use to integrate \eqref{FactorPDEBackward} backward in time to predict $\varphi_{\varepsilon}(\cdot,t=t_0)$. Subsequently, applying the known boundary condition \eqref{factorBC} at $t=t_0$ generates a new guess for $\varphi_{\varepsilon,0}(\cdot)$, thereby completing one pass of the proposed recursion. This process is then repeated as shown in Fig. \ref{fig:FixedPointRecursion}.

\begin{figure}[hbt!]
\centering
\includegraphics[width=0.7\textwidth]{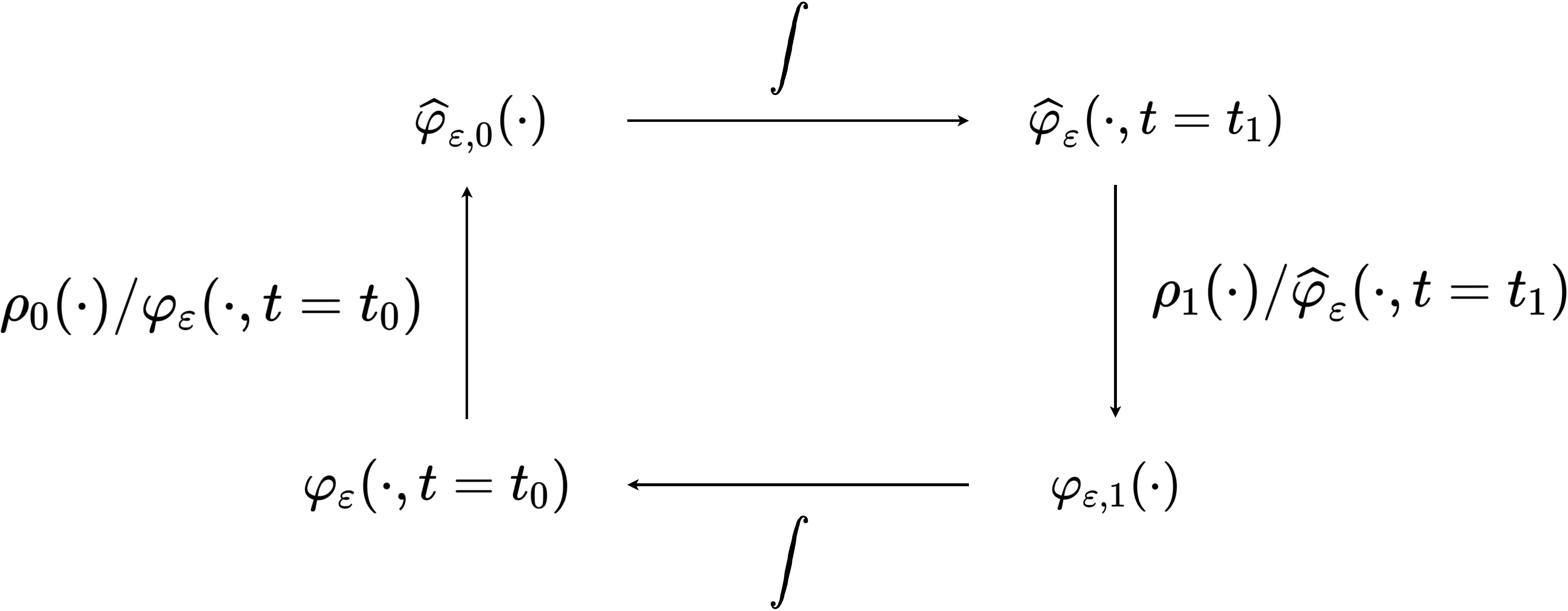}
\caption{Fixed point recursion over the endpoint Schr\"{o}dinger factor pair $(\widehat{\varphi}_{\varepsilon,0},\varphi_{\varepsilon,1})$ to solve the boundary-coupled system of PDEs \eqref{FactorPDEs}-\eqref{factorBC}. The arrows with integrals denote marching the respective IVP solution over time.}
\label{fig:FixedPointRecursion}
\end{figure}

Notice that in practical numerical simulations, $\rho_0,\rho_1$ are compactly supported, which is indeed a sufficient condition for our standing assumption $\rho_0,\rho_1\in\mathcal{P}_2\left(\mathbb{R}^3\right)$. That the above-mentioned recursion for compactly supported endpoint data, has \emph{guaranteed linear convergence} w.r.t. the Hilbert's projective metric \cite{hilbert1895gerade,bushell1973hilbert,franklin1989scaling}, has been proved in various degrees of generalities in the literature, see e.g., \cite[Sec. III]{chen2016entropic}. For references related to this contractive fixed point recursion, see e.g., \cite{de2021diffusion,pavon2021data,caluya2021reflected,caluya2021wasserstein,10293168}. 

We note that a convergent solution for the Schr\"{o}dinger factor recursion determines these factors in a projectivized unique sense, i.e., the recursion finds $\left(\kappa\widehat{\varphi}_{\varepsilon,0}(\cdot), \frac{1}{\kappa}\varphi_{\varepsilon,1}(\cdot)\right)$, and thus $\left(\kappa\widehat{\varphi}_{\varepsilon}(\cdot,t),\frac{1}{\kappa}\varphi_{\varepsilon}(\cdot,t)\right)$ for $t\in[t_0,t_1]$, up to an arbitrary constant $\kappa>0$. The product of these factors is the optimally controlled joint PDF $\rho^{\text{opt}}_{\varepsilon}(\cdot,t)$, which is unique in the usual sense. 

For the unknown function pair \eqref{defEndpointFactors}, the recursion proposed above effectively solves a \emph{Schr\"{o}dinger system}, i.e., a system of nonlinear integral equations
\begin{subequations}
\begin{align}
&\rho_0(\boldsymbol{r})=\widehat{\varphi}_{\varepsilon,0}(\boldsymbol{r}) \int_{\mathbb{R}^3} q_{\varepsilon}(t_0, \boldsymbol{r}, t_1, \boldsymbol{y}) \varphi_{\varepsilon,1}(\boldsymbol{y}) \mathrm{d} \boldsymbol{y}, \\
&\rho_1(\boldsymbol{r})=\varphi_{\varepsilon,1}(\boldsymbol{r}) \int_{\mathbb{R}^3} q_{\varepsilon}(t_0, \boldsymbol{y}, t_1, \boldsymbol{r}) \widehat{\varphi}_{\varepsilon,0}(\boldsymbol{y}) \mathrm{d} \boldsymbol{y},
\end{align}
\label{SchrodingerSystem}
\end{subequations}
where $q_{\varepsilon}$ is the (uncontrolled) Markov kernel associated with \eqref{FactorPDEs}. In our setting, evaluation of the integrals in \eqref{SchrodingerSystem} are performed by the IVP solvers for \eqref{FactorPDEs}. For the readers' convenience, we summarize the steps of our proposed algorithm. 

\medskip

    \textbf{Step 1.}  Make an initial guess for $\widehat{\varphi}_{\varepsilon,0}$ that is everywhere positive.

    \medskip
    
    \textbf{Step 2.} Use the $\widehat{\varphi}_{\varepsilon,0}$ from \textbf{Step 1} to integrate \eqref{FactorPDEForward} from $t = t_0$ to $t = t_1$, to determine $\widehat{\varphi}_{\varepsilon}(\cdot,t = t_1)$.

    \medskip
    
    \textbf{Step 3.} Set $\varphi_{\varepsilon,1} = \rho_1(\cdot) / \widehat{\varphi}_{\varepsilon}(\cdot,t = t_1)$ by enforcing \eqref{factorBC} at $t=t_1$.

    \medskip
    
    \textbf{Step 4.} Use the $\varphi_{\varepsilon,1}$ from \textbf{Step 3} to integrate \eqref{FactorPDEBackward} from $t = t_1$ to $t = t_0$ to determine $\varphi_{\varepsilon}(\cdot,t = t_0)$.

    \medskip
    
    \textbf{Step 5.} Redefine $\widehat{\varphi}_{\varepsilon,0} = \rho_0 / \varphi_{\varepsilon}(\cdot,t = t_0)$ by enforcing \eqref{factorBC} at $t=t_0$.

    \medskip
    
    \textbf{Step 6.} Repeat \textbf{Steps 2-5} until the pair $(\widehat{\varphi}_{\varepsilon,0}(\cdot),\varphi_{\varepsilon,1}(\cdot))$ has converged up to a desired numerical tolerance.

    \medskip

\textbf{Step 7.} With the converged endpoint factors $\widehat{\varphi}_{\varepsilon,0}(\cdot), \varphi_{\varepsilon,1}(\cdot)$ from \textbf{Step 6}, compute the transient factors
$\widehat{\varphi}_{\varepsilon}(\cdot,t), \varphi_{\varepsilon}(\cdot,t)$ at desired $t\in[t_0,t_1]$ using the IVP solver for \eqref{FactorPDEForward} and \eqref{FactorPDEBackward}. 

\medskip

\textbf{Step 8.} Use
\eqref{RecoverOptimalSolution} to return the optimal solution $\left(\rho_{\varepsilon}^{\rm{opt}},\bm{v}_{\varepsilon}^{\rm{opt}}\right)$ for the L-SBP.

\medskip
In describing our overall approach, so far we presumed the availability of an IVP solver for \eqref{FactorPDEs}. We next detail an integral representation formula for the solution of the IVP for \eqref{FactorPDEs}. This in turn facilitates the implementation of such an IVP solver.

\subsection{Integral representation for the Schr\"{o}dinger factors}\label{subsec:IntegralRepresentation}
We need the following Lemma \ref{LemmaReactionDiffusionPDESolRn}. Its proof uses basic Fourier transform results collected in Appendix A. The proof is deferred to Appendix B.

\begin{lemma}[Fredholm Integral representation for the solution of linear reaction-diffusion PDE IVP with state-dependent reaction rate]\label{LemmaReactionDiffusionPDESolRn}For $t\in[t_0,\infty)$, the solution of the reaction-diffusion PDE IVP
\begin{align}
\dfrac{\partial u}{\partial t} = a\Delta_{\bm{x}} u + b(\bm{x}) u, \quad \bm{x}\in\mathbb{R}^{d}, \quad u(\bm{x},t=t_0)=u_0(\bm{x})\;\text{given},
\label{PDEIVP}
\end{align}
where the constant $a>0$ and the function $b(\bm{x})$ is sufficiently smooth a.e., satisfies
\begin{align}
u(\bm{x},t) = &\underbrace{\dfrac{1}{\sqrt{\left(4\pi a t \right)^{d}}}\displaystyle\int_{\mathbb{R}^{d}}\exp\left(-\dfrac{\vert\bm{x}-\bm{y}\vert^2}{4at}\right) u_0(\bm{y})\differential\bm{y}}_{\text{term 1}} \nonumber\\
&+ \underbrace{\displaystyle\int_{t_0}^{t} \dfrac{1}{\sqrt{\left(4\pi a (t-\tau) \right)^{d}}}\displaystyle\int_{\mathbb{R}^{d}}\exp\left(-\dfrac{\vert\bm{x}-\bm{y}\vert^2}{4a(t-\tau)}\right)b(\bm{y})u(\bm{y},\tau)\:\differential\bm{y}\:\differential\tau}_{\text{term 2}}.
\label{IntegralRepresententationFormula}
\end{align}
\end{lemma}
\begin{remark}\label{TwoSummandsRemark}
The right-hand-side of \eqref{IntegralRepresententationFormula} is a sum of two terms: term 1 is the solution of the heat PDE IVP accounting for the initial condition $u_0(\cdot)$; term 2 captures the effect of the reaction rate $b(\cdot)$. We can view \eqref{IntegralRepresententationFormula} as a Fredholm integral equation of second kind \cite[Ch. 2]{tricomi1985integral} over the joint $(\bm{x},t)$ coordinates.
\end{remark}
Specializing \eqref{IntegralRepresententationFormula} with $n=3$ for the IVP involving \eqref{FactorPDEForward}, we get an integral representation formula
\begin{align}
\label{eq:int_u}
\widehat{\varphi}_{\varepsilon}(\bm{r},t) =  &\frac{1}{\sqrt{(4 \pi \varepsilon t)^3}} \int_{\mathbb{R}^3} \textrm{exp} \left( - \frac{\vert \bm{r} - \tilde{\bm{r}} \vert^2}{4 \varepsilon t} \right) \widehat{\varphi}_{\varepsilon,0}(\tilde{\bm{r}})  \:\differential\tilde{\bm{r}} \nonumber\\
&+ \int_{t_0}^t \frac{1}{2\varepsilon\sqrt{(4 \pi \varepsilon (t - \tau))^3}} \int_{\mathbb{R}^3} \textrm{exp} \left( - \frac{\vert \bm{r} - \tilde{\bm{r}} \vert^2}{4 \varepsilon(t-\tau)} \right) V(\tilde{\bm{r}}) \widehat{\varphi}_{\varepsilon} (\tilde{\bm{r}},\tau) \: \differential \tilde{\bm{r}} \: \differential \tau.
\end{align}
To complete \textbf{Step 2} of our algorithm, we approximate \eqref{eq:int_u} as further detailed in \cref{subsec:Riemann}.

Additionally, consider a change of time variable $t\mapsto \bar{t} := t_0 + t_1 - t$. When the physical time $t\in[t_1,t_0]$ flows backward, then the transformed time $\bar{t}\in[t_0,t_1]$ flows forward. This allows the IVP involving \eqref{FactorPDEBackward} to be rewritten as
\begin{align}
\frac{\partial \bar{\varphi}_{\varepsilon}}{\partial \bar{t}}\left(\bm{r},\bar{t}\right) = \left(\varepsilon\Delta_{\bm{r}} + \frac{1}{2\varepsilon}V(\bm{r})\right)\bar{\varphi}_{\varepsilon}(\bm{r},\bar{t}), \qquad \bar{\varphi}_{\varepsilon}\left(\bm{r},\bar{t} = t_0\right) = \varphi_{\varepsilon,1}.
\label{eq:new_pde}
\end{align}
From the solution of \eqref{eq:new_pde}, we recover $\varphi_{\varepsilon}(\bm{r},t)=\bar{\varphi}_{\varepsilon}(\bm{r},t_{0}+t_{1}-t)$. Thus, for our algorithm in \cref{subsec:OverallApproach}, we can reuse the same IVP solver from \textbf{Step 2} to complete \textbf{Step 4}.

\subsection{Left Riemann sum approximation}\label{subsec:Riemann}
For the numerical implementation of the IVP solver mentioned before, we now explain how the right-hand-side of \eqref{eq:int_u}, specifically its second summand, can be approximated via left Riemann sum. Recall that (cf. Remark \ref{TwoSummandsRemark}) the first summand in the right-hand-side of \eqref{eq:int_u} is the well-known solution of the heat equation, denoted as $\widehat{\varphi}_{\text{heat}}(\bm{r},t)$, which can be implemented by direct matrix-vector multiplication over a computational domain. In particular, we uniformly discretize a hyper-rectangular computational domain $[x_{\min}, x_{\max}] \times [y_{\min}, y_{\max}] \times [z_{\min}, z_{\max}] \times [t_0, t_1] \subset \mathbb{R}^3 \times [t_0, t_1]$. Along the space-time dimensions, we use constant step sizes $\Delta x, \Delta y, \Delta z$, $\Delta t$ for $N_x$, $N_y$, $N_z$, $N_t$ steps, respectively.

We assign $\widehat{\varphi}_{\varepsilon}(\bm{r}, t) \leftarrow \widehat{\varphi}_{\varepsilon,0}$ for $t \in [t_0, t_0 + \Delta t)$, and approximate the second term in \eqref{eq:int_u} as 
\begin{align}
&\label{phihatleftfirst} \widehat{\varphi}_{\textrm{left}}(\bm{r}, t = t_0 + \Delta t) \approx \frac{\Delta x \Delta y \Delta z \Delta t}{2 \varepsilon \sqrt{(4 \pi \varepsilon (\Delta t))^3}}  \sum_{m=0}^{N_x} \sum_{n=0}^{N_y} \sum_{j=0}^{N_z} \frac{V(\widetilde{\bm{r}}_{(m, n, j)}) \widehat{\varphi}_{\varepsilon,0}( \widetilde{\bm{r}}_{(m, n, j)})}{\textrm{exp}\left( \frac{\vert \bm{r} - \widetilde{\bm{r}}_{(m, n, j)}\vert^2}{4\varepsilon \Delta t} \right)} ,
\end{align}
where $\widetilde{\bm{r}}_{(m, n, j)} := [x_{\min} + m \Delta x, y_{\min} + n \Delta y, z_{\min} + j\Delta z]^{\top}$ and the tuple $(m,n,j)\in \{0,1,\hdots,N_x\}\times\{0,1,\hdots,N_y\}\times\{0,1,\hdots,N_z\}$.
We thus conclude that an appropriate approximation of $\widehat{\varphi}_{\varepsilon}(\bm{r}, t = t_0 + \Delta t)$ is
\begin{align*}
\widehat{\varphi}_{\varepsilon,\textrm{approx}}(\bm{r}, t = t_0 + \Delta t) = \widehat{\varphi}_{\textrm{heat}}(\bm{r}, t = t_0 + \Delta t) + \widehat{\varphi}_{\textrm{left}}(\bm{r}, t = t_0 + \Delta t).    
\end{align*}
For $k\in\{1,\hdots,N_t\}$, we construct a recursive formula for $\widehat{\varphi}_{\varepsilon,\textrm{approx}}(\bm{r}, t = t_0 + k\Delta t)$ using $\widehat{\varphi}_{\varepsilon,\textrm{approx}}(\bm{r}, t = t_0 + q\Delta t)$ $\forall \{q \in \mathbb{N}\cup\{0\} \mid q < k\}$. We approximate $\widehat{\varphi}_{\varepsilon}(\bm{r}, t)\approx\widehat{\varphi}_{\varepsilon,\textrm{approx}}(\bm{r}, t = t_0 + q\Delta t)$ for $t \in [t_0+q\Delta t, t_{0}+(q+1)\Delta t)$. To evaluate the second term of \eqref{eq:int_u}, we integrate over each interval $[t_0 + q\Delta t, t_0 + (q+1)\Delta t)$ using triple summation, as in the calculation of $\widehat{\varphi}_{\textrm{left}}(\bm{r}, t = t_0 + \Delta t)$. This yields the recursive formula
\begin{align}
\widehat{\varphi}_{\textrm{left}}(\bm{r}, t_0+k\Delta t) = \sum_{q = 0}^{k-1} \frac{\Delta x \Delta y \Delta z \Delta t}{2 \varepsilon \sqrt{(4 \pi \varepsilon (k-q)\Delta t)^3}} \sum_{m=0}^{N_x}  \sum_{n=0}^{N_y}  \sum_{j=0}^{N_z}  \frac{V(\widetilde{\bm{r}}_{(m, n, j)}) \widehat{\varphi}_{\varepsilon,\textrm{approx}}^{(m, n, j, q)}}{\textrm{exp} \left( \frac{\vert \bm{r} - \widetilde{\bm{r}}_{(m, n, j)} \vert^2}{4\varepsilon(k-q)\Delta t} \right) },
\label{left_approx}
\end{align}
where $\widehat{\varphi}_{\varepsilon,\textrm{approx}}^{(m, n, j, q)} = \widehat{\varphi}_{\varepsilon,\textrm{approx}}(\widetilde{\bm{r}}_{(m, n, j)}, t_0 + q\Delta t)$. Notice that specializing \eqref{left_approx} for $k=1$ recovers \eqref{phihatleftfirst}, as expected.

As earlier, we set $\widehat{\varphi}_{\varepsilon,\textrm{approx}}(\bm{r}, t = t_0 + k\Delta t) = \widehat{\varphi}_{\textrm{heat}}(\bm{r}, t =t_0 + k\Delta t) + \widehat{\varphi}_{\textrm{left}}(\bm{r}, t = t_0 + k\Delta t)$. Applying this result, we complete \textbf{Step 2} of our algorithm in \cref{subsec:OverallApproach} using 
\begin{align*}    \widehat{\varphi}_{\varepsilon}(\bm{r}, t_1) \approx \widehat{\varphi}_{\varepsilon,\textrm{approx}}(\bm{r}, t_1) = \widehat{\varphi}_{\textrm{heat}}(\bm{r}, t_1) + \widehat{\varphi}_{\textrm{left}}(\bm{r}, t_1).
\end{align*}
Likewise, \textbf{Step 4} of our algorithm in \cref{subsec:OverallApproach} is completed via this left Riemann sum approximation in conjunction with the change of variable $\bar{t} = t_0 + t_1 - t$ discussed in \cref{subsec:IntegralRepresentation}.

\begin{remark}\label{Remark:FeynmanKac}
An alternative way to numerically solve the IVPs associated with \eqref{FactorPDEs} is to use the Feynman-Kac path integral \cite[Ch. 8.2]{oksendal2013stochastic}, \cite{yong1997relations}, \cite[Ch. 3.3]{yong1999stochastic}, which was recently used to solve a class of SBPs \cite{nodozi2022schr}. Specifically, the Feynman-Kac formula expresses the solution of \cref{FactorPDEBackward}  as an expectation:
\begin{align}
\varphi_{\varepsilon}(\bm{r},t) = \mathbb{E}\left[\varphi_{\varepsilon,1}\left(\bm{r}_1\right)\exp\left(\int_{t}^{t_{1}}V\left(\bm{r}_s\right)\differential s\right)\bigg\vert \bm{r}_t = \bm{r}\right].
\label{FeynmanKacBackwardReactionDiffusion}    
\end{align}
Note that this also requires a function approximation oracle.
\end{remark}


\section{Numerical results}\label{sec:Numerical}
We illustrate our proposed algorithm for solving the L-SBP in a low Earth orbit transfer case study as in \cite{kim2020optimal} and \cite{curtis2013orbital}. Specifically, we consider stochastic transfer of a spacecraft from an initial mean position $[5000, 10000, 2100 ]^{\top}$ km to a final mean position $[-14600, 2500, 7000]^{\top}$ km in Earth centered inertial coordinates for a fixed flight time horizon $[t_0,t_1]=[0,1\:\text{hour}]$.

For numerical conditioning, we re-scale the variables as $\bm{r}^{\prime} := \bm{r}/R$, $t^{\prime} := t/T$, where the normalization constants $R = 6600$ km and $T = 5399$ s. In these re-scaled coordinates, the potential \eqref{defPotential} becomes
\begin{align}
    \overline{V}(\bm{r}^{\prime}) &= - \frac{\mu_{\textrm{new}}R}{T| \bm{r}^{\prime}|} - \frac{\mu_{\textrm{new}} J_2 R_{\textrm{Earth}}^2}{2 R T |\bm{r}^{\prime}|^3} \left( 1 - \frac{3 R^2 (z^{\prime})^2}{R^2|\bm{r}^{\prime}|^2} \right) \nonumber\\
    &= - \frac{\mu_{\textrm{new}}R}{T| \bm{r}^{\prime}|}\left( 1+\frac{ J_2 R_{\textrm{Earth}}^2}{2 R |\bm{r}^{\prime}|^2} \right) \left( 1 - \frac{3 (z^{\prime})^2}{|\bm{r}^{\prime}|^2} \right), \quad \mu_{\textrm{new}} := \frac{\mu T}{R^2}\:\text{km/s}.
\label{newPotential}
\end{align}

To ensure that all controlled state sample paths stay above the surface of Earth, we add an regularizer to the potential $\overline{V}$, and write the regularized version of \eqref{newPotential} as
\begin{align}    
\overline{V}_{\textrm{reg}}(\bm{r}^{\prime}) = - \frac{\mu_{\textrm{new}}R}{T\vert \bm{r}^{\prime}\vert}\!\left(\!1+\frac{ J_2 R_{\textrm{Earth}}^2}{2 R \vert\bm{r}^{\prime}\vert^2}\!\right)\!\! \left( 1 - \frac{3 (z^{\prime})^2}{\vert \bm{r}^{\prime}\vert^2} \right) \!+ \!\gamma\! \left( \vert \bm{r}^{\prime}\vert^2 - \frac{(R_{\rm{Earth}} +c)^2}{R^2}\right),
\label{newPotentialWithReg}
\end{align}
where the constant $c$ is chosen in our simulation such that $R_{\rm{Earth}} +c = 6560$ Km, and $\gamma>0$ is the regularization weight.

Using the pushforward of probability measures, the initial and final PDFs in the new coordinates become 
\begin{align}    \rho_{0}^{\prime}(\bm{r}^{\prime}) = R^3 \rho_0(R \bm{r}^{\prime}), \quad \rho_{1}^{\prime}(\bm{r}^{\prime}) = R^3 \rho_1(R \bm{r}^{\prime}).
    \label{eq:rho_var}
\end{align}
Letting $\widehat{\phi}(\bm{r}^{\prime}, t^{\prime}):= \widehat{\varphi}_{\varepsilon}(\bm{r}, t)$, and noting that 
$\frac{\partial \widehat{\varphi}_{\varepsilon}}{\partial t}(\bm{r}, t) = \frac{1}{T} \frac{\partial \widehat{\phi}}{\partial t^{\prime}}(\bm{r}^{\prime}, t^{\prime})$, $\Delta_{\bm{r}} \widehat{\varphi}_{\varepsilon} (\bm{r}, t) = \frac{1}{R^2} \Delta_{\bm{r}^{\prime}} \widehat{\phi} (\bm{r}^{\prime}, t^{\prime})$,
we rewrite \eqref{FactorPDEForward} as
\begin{align}
    \frac{\partial \widehat{\phi}}{\partial t^{\prime}}(\bm{r}^{\prime}, t^{\prime}) = \left( \frac{\varepsilon T}{R^2} \Delta_{\bm{r}^{\prime}} + \frac{T}{2\varepsilon} \overline{V}_{\textrm{reg}}(\bm{r}^{\prime}) \right) \widehat{\phi} (\bm{r}^{\prime}, t^{\prime}),
\label{FactorPDEForwardNewVariables}
\end{align}
and apply the left Riemann sum approach from \cref{subsec:Riemann} to solve the corresponding PDE IVP. Likewise, we apply a change of variable $\phi(\bm{r}^{\prime}, t^{\prime}) := \varphi_{\varepsilon}(\bm{r}, t)$ to \eqref{FactorPDEBackward}.

We choose the endpoint PDFs as trivariate normal PDFs $\rho_0 = \mathcal{N}(\bm{\mu}_0, \bm{\Sigma}_0)$, $\rho_1 = \mathcal{N}(\bm{\mu}_1, \bm{\Sigma}_1)$ with respective mean vectors $\bm{\mu}_0 = [5000, 10000, 2100 ]^{\top},\bm{\mu}_1 = [-14600, \allowbreak 2500, 7000]^{\top}$ as in \cite{kim2020optimal,curtis2013orbital}, and respective covariance matrices $\bm{\Sigma}_0 = {\rm{diag}}\left(\bm{\mu}_0^2\right)/100, \allowbreak\bm{\Sigma}_1 = {\rm{diag}}\left(\bm{\mu}_1^2\right)/100$, where the vector exponentiation is element-wise. Using \eqref{eq:rho_var}, we implement the proposed algorithm in the re-scaled coordinates $(\bm{r}', t^{\prime})$. Note that while we use $\rho_0,\rho_1$ as Gaussians for this specific simulation, our method applies for any non-Gaussian PDFs with finite second moments.

We consider the spatial grid $\left[-\frac{30000}{R}, \frac{10000}{R}\right] \times \left[-\frac{5000}{R}, \frac{25000}{R}\right] \times \left[-\frac{5000}{R}, \frac{25000}{R}\right]$ and set $N_x = N_y = N_z = 12$, $N_t = 50$, $\varepsilon = 15000$, $\gamma = 7.5$. We perform the Schr\"{o}dinger factor recursion (\textbf{Step 6} in \cref{subsec:OverallApproach}) for 7 iterations which we found sufficient for convergence. The optimally controlled joint $\rho^{\prime {\rm{opt}}}_{\varepsilon}(\bm{r}^{\prime}, t^{\prime})$ computed from \textbf{Step 8} in \cref{subsec:OverallApproach}, is then transformed back to
$\rho^{{\rm{opt}}}_{\varepsilon}(\bm{r}, t) = \frac{1}{R^3} \rho^{\prime} \left(\bm{r}^{\prime}, t^{\prime}\right) = \frac{1}{R^3} \rho^{\prime} \left(\bm{r}/R, t/T\right)$. Likewise, we obtain the optimal control $\bm{v}_{\varepsilon}^{\rm{opt}}(\bm{r}, t) = 2 \varepsilon \nabla_{\bm{r}} \log\varphi(\bm{r},t)=\frac{2 \varepsilon}{R} \nabla_{\bm{r}^{\prime}} \log\phi(\bm{r}^{\prime}, t^{\prime})$. We evaluate $\left(\rho_{\varepsilon}^{\rm{opt}},\bm{v}_{\varepsilon}^{\rm{opt}}\right)$ on a grid of size $100\times 100\times 100\times N_t$ which is denser than the original grid of size $N_x \times N_y \times N_z \times N_t$.

To evaluate the effectiveness of our proposed solution, we perform a closed loop simulation using the computed optimal policy $\bm{v}_{\varepsilon}^{\rm{opt}}(\bm{r}, t)$. We draw $50$ samples from the known $\rho_0$, and use the Euler-Maruyama scheme to propagate these samples forward in time from $t_0 = 0$ s to $t_1 = 3600$ s, using the closed loop It\^{o} SDE \eqref{SDESamplePath}. As the values of $\bm{v}_{\varepsilon}^{\rm{opt}}(\bm{r}, t)$ are known only on the $100\times 100\times 100\times N_t$ grid, we use the nearest neighbor approximation to query $\bm{v}_{\varepsilon}^{\rm{opt}}$ at a out-of-grid sample during SDE integration. The 50 optimally controlled closed-loop state sample paths thus computed, are shown in Fig. \ref{fig:sample_paths}, demonstrating the transfer of the controlled stochastic state from $\rho_0$ to $\rho_1$.
\begin{figure}[t] 
\centering
\includegraphics[width=.7\textwidth]{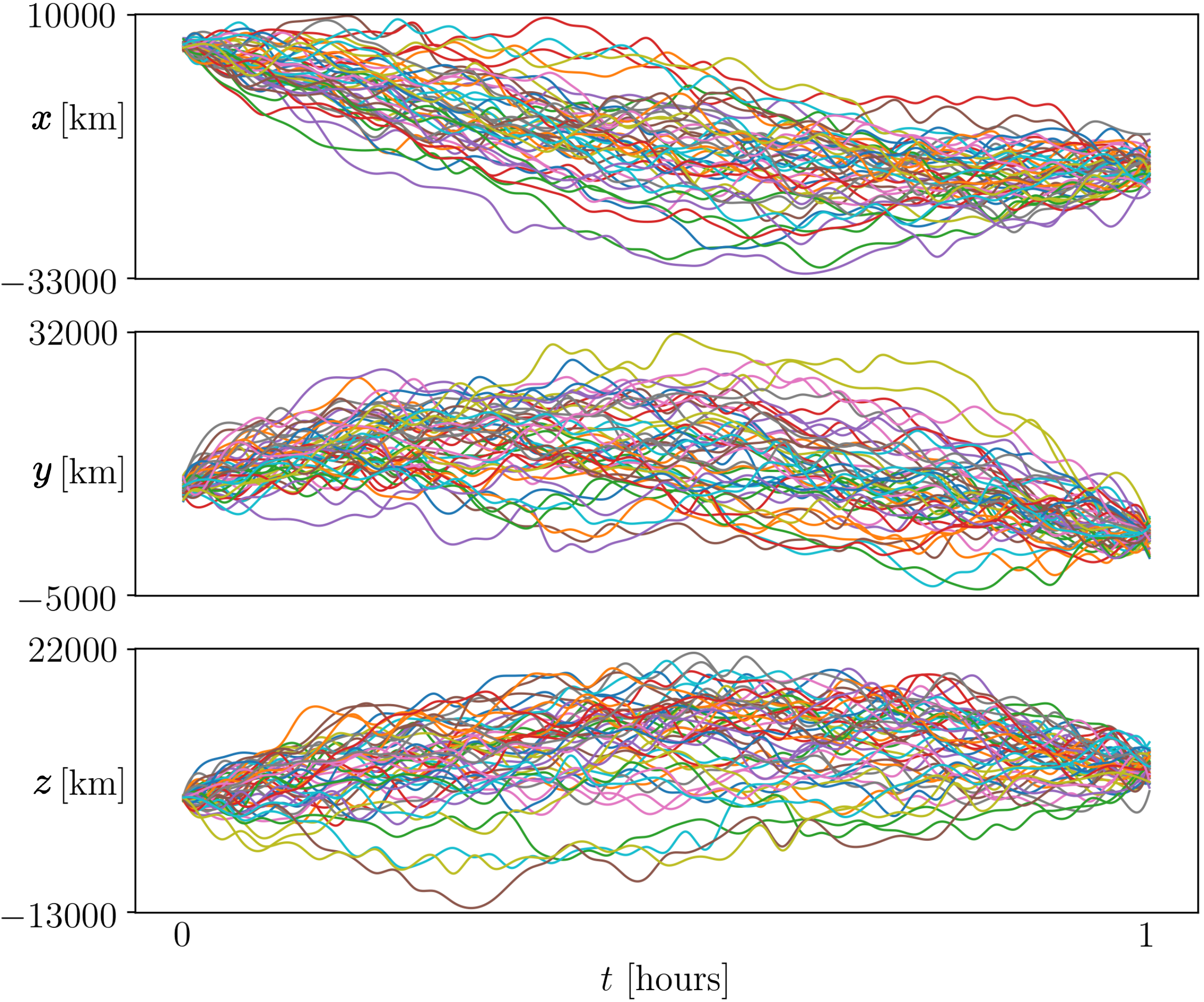}
\caption{Optimally controlled closed loop state sample paths for the numerical simulation in \cref{sec:Numerical}.}
\label{fig:sample_paths}
\end{figure}
\begin{figure}[hbt!]
\centering
\includegraphics[width=.7\textwidth]{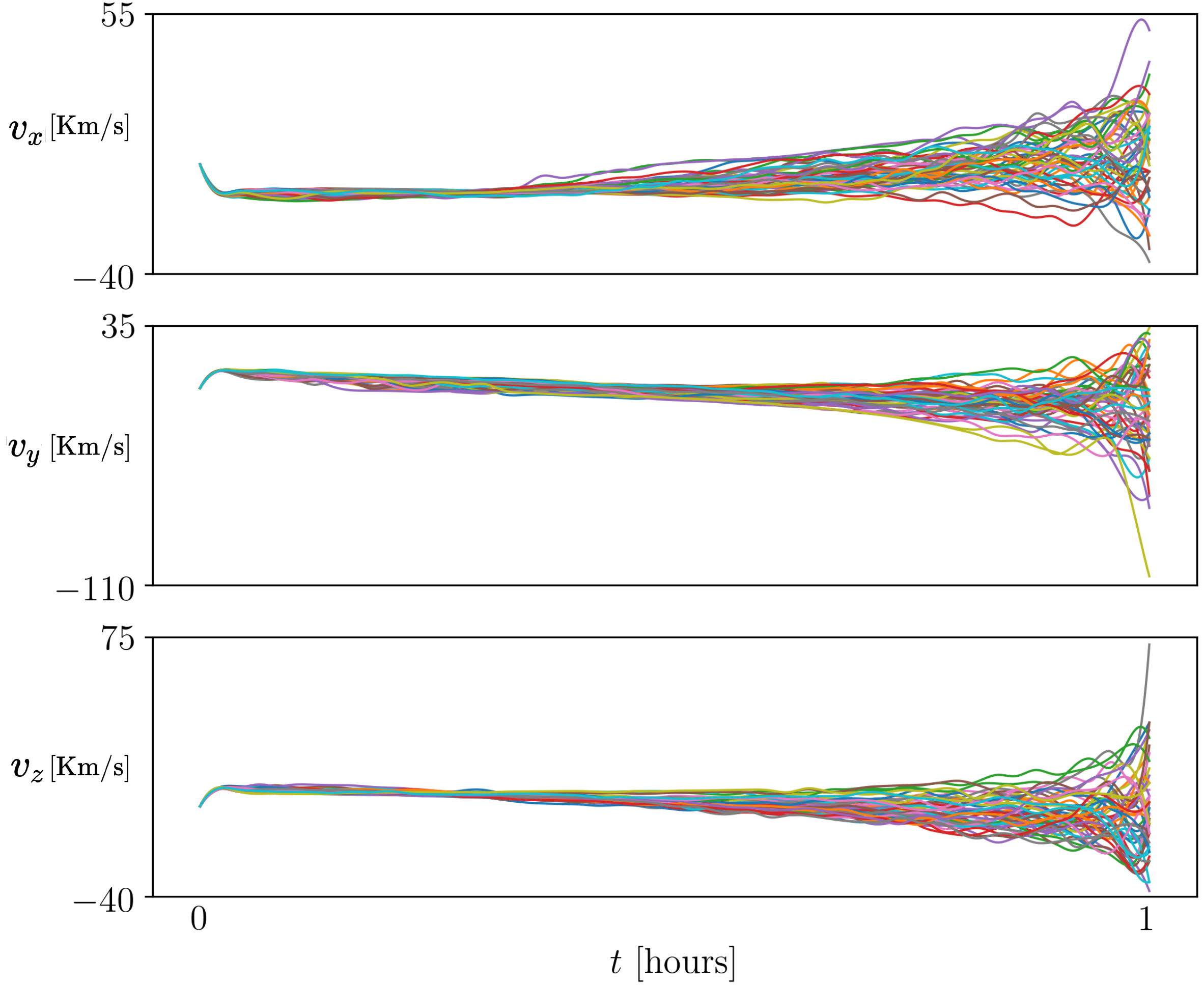}
\caption{Sample paths for the components $v_x,v_y,v_z$ of the optimal control for the numerical simulation in \cref{sec:Numerical}.}
\label{fig:control}
\end{figure}
Fig. \ref{fig:control} shows the corresponding sample paths of the optimal controls. 

\begin{figure}[t] 
\centering
\includegraphics[width=.4\textwidth]{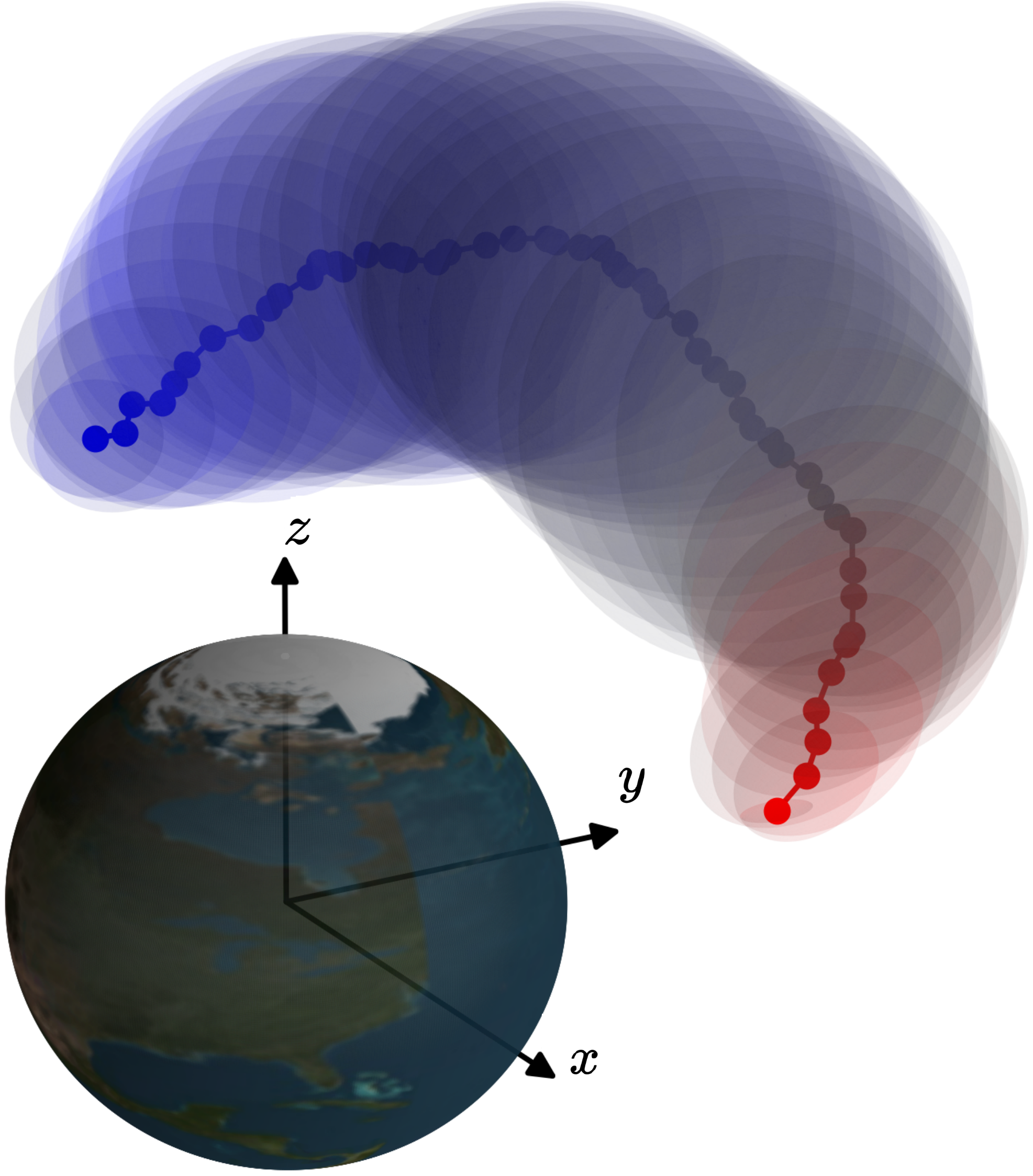}
\caption{The filled circles show the mean position snapshots for the 50 optimally controlled sample paths in $\mathbb{R}^3$, as detailed in \cref{sec:Numerical}. The translucent ellipsoids denote one standard deviation around each mean position.}
\label{fig:3D_plot}
\end{figure}
In Fig. \ref{fig:3D_plot}, the filled circles depict the evolution (\emph{red} initial and \emph{blue} final) of the means of the aforesaid 50 optimally controlled sample path ensemble in $\mathbb{R}^3$. The translucent ellipsoids therein display one standard deviations around the means. We include the Earth centered inertial coordinate in Fig. \ref{fig:3D_plot} for reference, and to emphasize that our regularizer in \eqref{newPotentialWithReg} successfully keeps the optimally controlled sample paths above the Earth's surface. 

Fig. \ref{fig:marginals} plots five snapshots for the univariate $x,y,z$ position marginals computed from the optimally controlled joint PDF $\rho_{\varepsilon}^{\text{opt}}(\bm{r},t)$. This again highlights successful transfer of the stochastic state from given $\rho_0$ to given $\rho_1$ over the given flight time horizon $[t_0,t_1]$. From Fig. \ref{fig:marginals}, we also note that the optimally controlled marginals (and thus the joint PDFs) at the intermediate times are non-Gaussians even though the endpoints $\rho_0,\rho_1$ are Gaussians in this simulation case study. 
\begin{figure}
    \centering
    \begin{subfigure}{0.32\textwidth}
        \includegraphics[width=\textwidth]{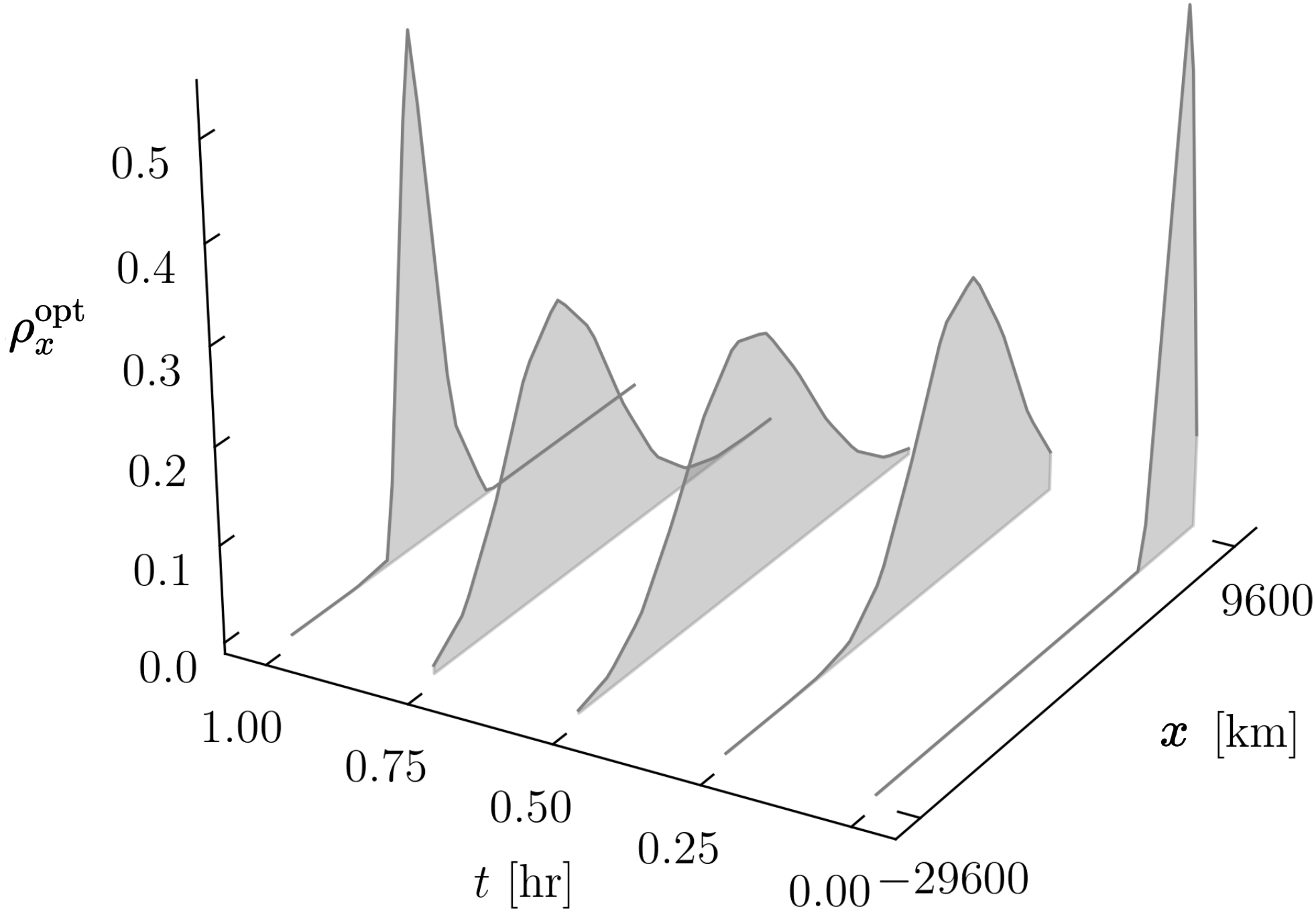}
    \end{subfigure}
    \begin{subfigure}{0.32\textwidth}
        \includegraphics[width=\textwidth]{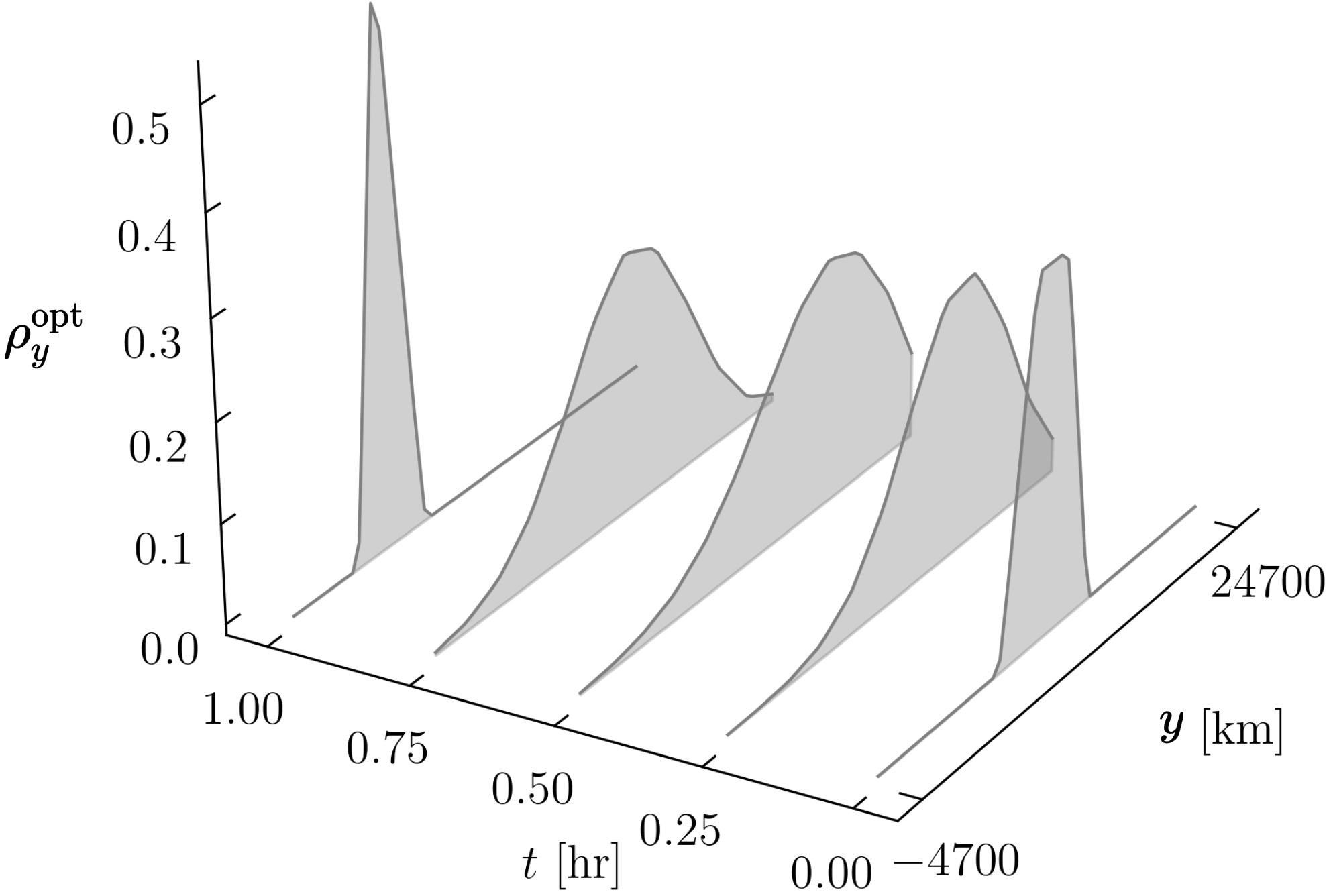}
    \end{subfigure}
    \begin{subfigure}{0.32\textwidth}
        \includegraphics[width=\textwidth]{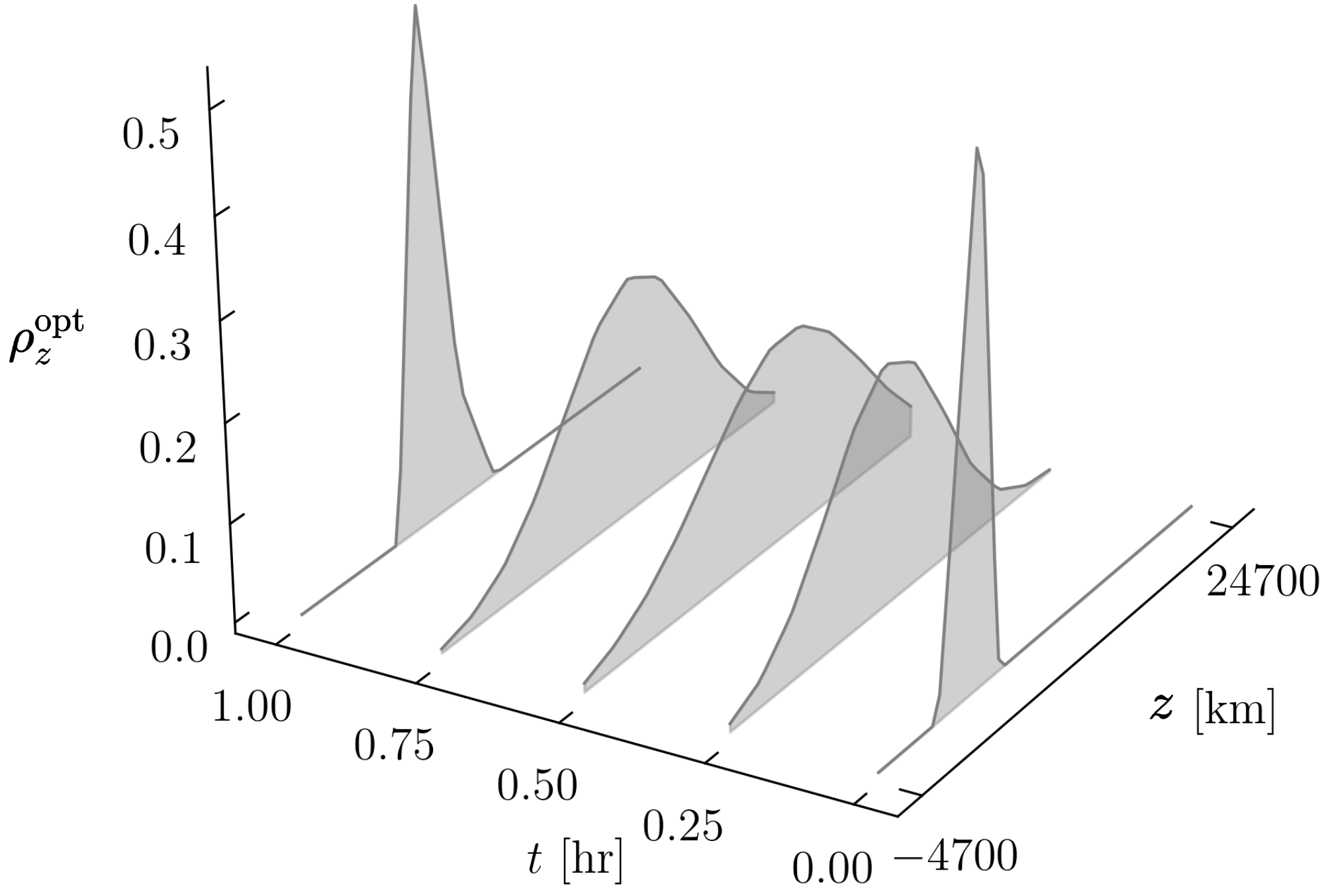}
    \end{subfigure}
    \caption{Univariate $x,y,z$ marginals for the optimally controlled joint $\rho_{\varepsilon}^{\text{opt}}(\bm{r},t)$ at $t = 0, 0.25, 0.50, 0.75, 1$ hours.}
    \label{fig:marginals}
\end{figure}


\section{Conclusion}
This work clarifies connections between the probabilistic Lambert problem--a problem in orbital mechanics--and the OMT and SBP, which are topics of current interest in applied mathematics, stochastic control and machine learning. We showed that the probabilistic Lambert problem is, in fact, a generalized OMT where the gravitational potential plays the role of a state cost. Using Figalli's theory, we established the existence-uniqueness of solution for this problem. In the presence of stochastic process noise, this problem is shown to become a generalized SBP for which we derived a large deviation principle, thereby proving the existence-uniqueness for this stochastic variant. We derived the associated conditions for optimality, and reduced the same to a boundary-coupled system of reaction-diffusion PDEs where the gravitational potential plays the role of a state-dependent reaction rate. Building on these newfound connections, we presented a novel algorithm and an illustrative numerical case study to demonstrate the solution of the probabilistic Lambert problem via nonparametric computation. Our methods should be of broad interest for solving guidance and navigation problems with stochastic uncertainties. For the readers' convenience, we conclude with a diagrammatic relation chart between different problems studied in this paper.
\[ \begin{tikzcd}[column sep=large]
\text{Probabilistic Lambert problem}\; \eqref{ProbLambertFeasibility} \;\iff \;\text{L-OMT}\;\eqref{LambertOT}\quad \xLeftarrow{\varepsilon\downarrow 0} \quad \text{L-SBP}\;\eqref{LambertSBP}\\
\qquad\qquad\qquad\qquad\qquad\qquad\qquad\qquad\qquad\qquad\bigg\Downarrow{V\equiv 0} \qquad\qquad\qquad\bigg\Downarrow{V\equiv 0}\\
\qquad\qquad\qquad\qquad\qquad\qquad\qquad\qquad\qquad\quad\text{OMT}\;\eqref{BenamouBrenierOMT}\quad \xLeftarrow{\varepsilon\downarrow 0} \quad \text{SBP}\;\eqref{SBP}
\end{tikzcd}
\]



\appendix
\section{Some results on Fourier transform} 
Let $\complexi:=\sqrt{-1}$. We consider the non-unitary angular frequency version of Fourier transform
\begin{align}
\widehat{f}(\bm{\omega}) := \int_{\mathbb{R}^{d}} f(\bm{x})\exp\left(-\complexi\langle\bm{\omega},\bm{x}\rangle\right)\differential\bm{x},
\label{DefFourierTransform}
\end{align}
and the corresponding inverse Fourier transform
\begin{align}
f(\bm{x}) := \dfrac{1}{(2\pi)^{d}}\int_{\mathbb{R}^{d}} \widehat{f}(\bm{\omega})\exp\left(+\complexi\langle\bm{\omega},\bm{x}\rangle\right)\differential\bm{\omega}.
\label{DefInvFourierTransform}
\end{align}
We symbolically write $\widehat{f}(\bm{\omega}) = \mathcal{F}\left[f(\bm{x})\right]$, $f(\bm{x}) = \mathcal{F}^{-1}\left[\widehat{f}(\bm{\omega})\right]$.
Denote the convolution of functions $f$ and $g$, as $(f * g)(\bm{x}) := \int_{\mathbb{R}^{d}}f(\bm{x}-\bm{y})g(\bm{y})\differential\bm{y}=\int_{\mathbb{R}^{d}}f(\bm{y})g(\bm{x}-\bm{y})\differential\bm{y}$. We recall the following basic results for a suitably smooth function $f(\bm{x})$ on $\mathbb{R}^{d}$. For brevity, we omit the proofs for these facts.

\noindent\puteqnum\label{ForierTransformOfLaplacian} (Fourier transform of Laplacian on $\mathbb{R}^{d}$) $\mathcal{F}\left[\Delta_{\bm{x}} f\right] = - \vert\bm{\omega}\vert^{2}\:\widehat{f}(\bm{\omega})$.

\noindent\puteqnum\label{Convolution} (Convolution) $\mathcal{F}\left[\left(f * g\right)(\bm{x})\right] = \widehat{f}(\bm{\omega})\widehat{g}(\bm{\omega})$, $\mathcal{F}\left[f (\bm{x}) g(\bm{x})\right] = \dfrac{1}{(2\pi)^{d}}\left(\widehat{f} * \widehat{g}\right)(\bm{\omega})$.    
      
\noindent\puteqnum \label{InverseFourierOfExpNegSOS} (Transform for exp-negative-square) $\mathcal{F}^{-1}\left[\exp\left(-a\vert\bm{\omega}\vert^2 t\right)\right] = \dfrac{\exp\left(-\frac{\vert\bm{x}\vert^2}{4at}\right)}{\sqrt{\left(4\pi a t\right)^d}}$.


\section{Proof of Lemma \ref{LemmaReactionDiffusionPDESolRn}}\label{AppSubsecProofOfLemmaReactionDiffusionPDESolRn}
We denote the Fourier transform of $u(\bm{x},t)$ w.r.t. $\bm{x}$ as $\widehat{u}(\bm{\omega},t)$ for all $t\in[t_0,\infty)$, and the Laplace transform of $\widehat{u}(\bm{\omega},t)$ w.r.t. $t$ as $\widehat{U}(\bm{\omega},s)$. Let $\mathcal{L}\left[\cdot\right], \mathcal{L}^{-1}\left[\cdot\right]$ denote the Laplace and inverse Laplace transform operators, respectively.

Taking the Fourier transform to both sides of \eqref{PDEIVP}, we obtain $\dfrac{\differential}{\differential t} \mathcal{F}\left[u(\bm{x},t)\right] = a \mathcal{F}\left[\Delta u\right] + \mathcal{F}\left[b(\bm{x})u\right]$, which implies
\begin{align}
\dfrac{\differential}{\differential t}\widehat{u}(\bm{\omega},t) = -a\vert\bm{\omega}\vert^2 \: \widehat{u}(\bm{\omega},t) + \dfrac{1}{(2\pi)^{d}}\left(\widehat{b}(\bm{\omega}) * \widehat{u}(\bm{\omega},t)\right),
\label{FTofPDE}
\end{align}
where we used \eqref{ForierTransformOfLaplacian} and \eqref{Convolution}. Taking the Laplace transform to both sides of \eqref{FTofPDE}, we get
\begin{align}
& s\widehat{U}(\bm{\omega},s) - \widehat{u}_0(\bm{\omega}) = -a\vert\bm{\omega}\vert^2 \:\widehat{U}(\bm{\omega},s) + \dfrac{1}{(2\pi)^{d}}\displaystyle\int_{\mathbb{R}^{d}}\widehat{b}(\bm{\omega}-\bm{\eta})\widehat{U}(\bm{\eta},s)\:\differential\bm{\eta}\nonumber\\
\Rightarrow&\quad \widehat{U}(\bm{\omega},s) = \dfrac{\widehat{u}_0(\bm{\omega})}{s + a\vert\bm{\omega}\vert^2} + \dfrac{1}{(2\pi)^{d}}\displaystyle\int_{\mathbb{R}^{d}}\widehat{b}(\bm{\omega}-\bm{\eta})\dfrac{\widehat{U}(\bm{\eta},s)}{s + a\vert\bm{\omega}\vert^2}\:\differential\bm{\eta}.
\label{LaplaceOfFourier}
\end{align}
Inverse Laplace transform of \eqref{LaplaceOfFourier} yields
\begin{align}
&\widehat{u}(\bm{\omega},t) = \widehat{u}_0(\bm{\omega})\exp\left(-a\vert\bm{\omega}\vert^2 t\right) + \dfrac{1}{(2\pi)^{d}}\displaystyle\int_{\mathbb{R}^{d}}\widehat{b}(\bm{\omega}-\bm{\eta})\:\mathcal{L}^{-1}\left[\dfrac{\widehat{U}(\bm{\eta},s)}{s + a\vert\bm{\omega}\vert^2}\right]\:\differential\bm{\eta} \nonumber\\
&= \widehat{u}_0(\bm{\omega})\exp\left(-a\vert\bm{\omega}\vert^2 t\right) + \dfrac{1}{(2\pi)^{d}}\displaystyle\int_{\mathbb{R}^{d}}\widehat{b}(\bm{\omega}-\bm{\eta})\left(\int_{0}^{t}\widehat{u}(\bm{\eta},\tau)\exp\left(-a\vert\bm{\omega}\vert^2 (t-\tau)\right)\differential\tau\!\right)\!\differential\bm{\eta}\nonumber,
\end{align}
where the last line used the convolution property for the Laplace transform, namely $\mathcal{L}^{-1}\left[F(s)G(s)\right] = \left(f * g\right)(t)$.

We next change the order of integration (using Fubini-Tonelli Theorem) for the space-time integral above, then apply the inverse Fourier transform to both sides, and use \eqref{Convolution} to arrive at
\begin{align}
&u(\bm{x},t) = \!\left(u_0 * \mathcal{F}^{-1}\!\left[\exp\left(-a\vert\bm{\omega}\vert^2 t\right)\right]\right)\!(\bm{x})\label{FubiniThenInverseFT}\\
&+ \dfrac{1}{(2\pi)^{d}}\displaystyle\int_{t_0}^{t}\!\!\left(\!\mathcal{F}^{-1}\!\left[\exp\left(-a\vert\bm{\omega}\vert^2 (t-\tau)\!\right)\right] * \mathcal{F}^{-1}\!\left[\! \left(\int_{\mathbb{R}^{d}}\widehat{b}(\bm{\omega}-\bm{\eta})\widehat{u}(\bm{\eta},\tau)\differential\bm{\eta}\right)\!\right]\right)(\bm{x})\:\differential\tau.
\nonumber
\end{align}
Using \eqref{Convolution} again, we notice that
$\mathcal{F}^{-1}\left[ \left(\int_{\mathbb{R}^{d}}\widehat{b}(\bm{\omega}-\bm{\eta})\widehat{u}(\bm{\eta},\tau)\differential\bm{\eta}\right)\right] = \left(2\pi\right)^{d}b(\bm{x})u(\bm{x},\tau)$, which together with \eqref{InverseFourierOfExpNegSOS}, simplifies \eqref{FubiniThenInverseFT} to \eqref{IntegralRepresententationFormula}. \hfill\qed


\bibliographystyle{siamplain}
\bibliography{references}
\end{document}